\documentclass[pdflatex,sn-mathphys-num]{sn-jnl}
\usepackage{graphicx}%
\usepackage{multirow}%
\usepackage{amsmath,amssymb,amsfonts}%
\usepackage{amsthm}%
\usepackage{mathrsfs}%
\usepackage[title]{appendix}%
\usepackage{xcolor}%
\usepackage{textcomp}%
\usepackage{manyfoot}%
\usepackage{booktabs}%
\usepackage{algorithm}%
\usepackage{algorithmicx}%
\usepackage{algpseudocode}%
\usepackage{listings}%
\usepackage{tikz}
\usepackage{subcaption}

\theoremstyle{thmstyleone}%
\newtheorem{theorem}{Theorem}
\newtheorem{proposition}[theorem]{Proposition}%
\newtheorem{lem}[theorem]{Lemma}

\theoremstyle{thmstyletwo}%

\theoremstyle{thmstylethree}%
\newtheorem{definition}{Definition}%

\theoremstyle{thmstylethree}%
\newtheorem{corollary}{Corollay}%

\begin{document}

\title[Nonresonance for problems involving $(p,q)$-Laplacian equations with nonlinear perturbations]{Nonresonance for problems involving $(p,q)$-Laplacian equations with nonlinear perturbations}

\author*[1]{\fnm{Emer} \sur{Lopera}}\email{edloperar@unal.edu.co}

\author[2]{\fnm{Nsoki} \sur{Mavinga}}\email{mavinga@swarthmore.edu}
\equalcont{These authors contributed equally to this work.}

\author[3]{\fnm{Diana} \sur{Sanchez}}\email{dmsanchezm@unal.edu.co}
\equalcont{These authors contributed equally to this work.}

\affil*[1]{\orgdiv{Departamento de Matem\'aticas}, \orgname{Universidad Nacional de Colombia}, \orgaddress{\street{Carrera 27 No 64-60}, \city{Manizales}, \postcode{500001}, \state{Caldas}, \country{Colombia}}}

\affil[2]{\orgdiv{Department of Mathematics and Statistics}, \orgname{Swarthmore College}, \orgaddress{\street{500 College avenue}, \city{Swarthmore}, \postcode{19081}, \state{Pennsylvania}, \country{USA}}}

\affil[3]{\orgdiv{Departamento de Matem\'aticas}, \orgname{Universidad Nacional de Colombia}, \orgaddress{\street{Carrera 27 No 64-60}, \city{Manizales}, \postcode{500001}, \state{Caldas}, \country{Colombia}}}

\abstract{
We study the solvability of $(p,q)$-Laplacian problems with nonlinear reaction terms and non-homogeneous Neumann boundary conditions. First, we provide a complete
description of the spectrum of the eigenvalue problem  involving the $(p,q)$-Laplacian  with  weights and a spectral parameter present in both the differential equation and on the boundary. Then, using  variational methods and critical point theory, we prove the existence of weak solutions for the nonlinear problem when the nonlinearities involved  remain asymptotically, in some sense,  below the first eigenvalue of  the $q$-Laplacian problem with weights and a spectral parameter present in both the differential equation and on the boundary.  We also establish an existence result for the nonlinear problem when the nonlinearities involved  remain asymptotically  below the first Steklov-Neumann eigenvalue-line, which is a line connecting the first Steklov  and  first Neumann eigenvalues for $q$-Laplacian problems  with weights and a spectral parameter present  either in the differential equation or on the boundary.
}



\maketitle

\section{Introduction}\label{sec1}

We are concerned with the existence of weak solutions to following $(p,q)$-Laplacian boundary value problem:
\begin{equation}\label{nonlinpde}
\left\{\begin{array}{lrl}
        &-\Delta_p u-\Delta_q u +c_1(x)|u|^{p-2}u+c_2(x)|u|^{q-2}u&= f(x,u) \qquad \text{in } \Omega  \\
        &(|\nabla u|^{p-2}+|\nabla u|^{q-2})\frac{\partial u}{\partial \nu}&=g(x,u)\qquad \text{on }  \partial\Omega,
\end{array}   \right. 
\end{equation}
where $\Omega$ is a bounded domain in $\mathbb{R}^N,\, N\ge 2,$ with smooth boundary $\partial \Omega$,  $\partial/\partial\nu=\nu\cdot\nabla$ is the outward (unit) normal derivative on $\partial \Omega$, and $\Delta_{\theta}$ denotes the $\theta$-Laplacian, $-\Delta_{\theta}(u)=-\text{div}(|\nabla u|^{\theta-2}\nabla
 u)$. The nonlinearities $f(x,u)$ and  $g(x,u)$ interact with the spectrum of the following $(p,q)$-Laplacian problem:

\begin{equation}\label{eig-prob_1}
\left\{\begin{array}{lrl}
&        -\Delta_p u-\Delta_q u +c_1(x)|u|^{p-2}u+c_2(x)|u|^{q-2}u&=\lambda m(x)|u|^{q-2}u \qquad \text{in } \Omega  \\
      &  (|\nabla u|^{p-2}+|\nabla u|^{q-2})\frac{\partial u}{\partial \nu}&=\lambda \rho(x)|u|^{q-2}u\qquad \text{on }  \partial\Omega ,
\end{array}   \right. 
\end{equation}

where the spectral parameter $\lambda$ is present in both the differential equation and on the boundary (with weights). Notice that when $\rho\equiv 0$ or $m\equiv 0$, problem \eqref{eig-prob_1} reduces to the eigenvalue problems under Neumann- or Steklov-type boundary conditions, respectively. These were considered in \cite{BLMG} in the case $m\equiv 0 $ and $\rho \equiv 1$, and in \cite{MM} for the case $m\equiv 1$, $\rho\equiv 0$, and $c_1$ or $c_2$ identically zero. In \cite{BLM} the authors provide a complete description of the spectrum of Problem \eqref{eig-prob_1} when $c_1\equiv 0$.  In \cite{BMG}, the authors described the set of eigenvalues of Problem \ref{eig-prob_1} for the case when both  $c_1\equiv c_2\equiv 0$. In this paper, we discuss the spectrum of Problem \eqref{eig-prob_1}  when $c_1$ and $c_2$ are not identically zero.

The main objective of this paper is twofold. In the first part, we provide a complete description of the spectrum of problem \eqref{eig-prob_1}  when $c_1$, $ c_2,$ and the weights $ m$ and $\rho$ are  not identically zero. This result can be considered, to some extent, as  an extension of the previous work in \cite {BLM,BLMG,mavinga,MNN,MM}. 
For more studies on eigenvalue problems involving $(p,q)$-Laplacian-type operators with linear boundary conditions, see, e.g. \cite{BLBAMG,BBMG, GJMORR,PNVCVF,PP} and references therein.

In the second part of the paper, based on our results from the spectral problem, the variational methods and the critical point theory, we prove the existence of weak solutions when the nonresonance condition occurs, that is, the nonlinearities involved remain  asymptotically, in some sense, below the first eigenvalue of  the $q$-Laplacian problem with weights $(m,\rho)$ and the spectral parameter $\lambda$ is  both in the differential equation and on the boundary. This first eigenvalue of  the $q$-Laplacian problem corresponds to the infimum of the set of eigenvalues of problem \eqref{eig-prob_1}. We also establish an existence result for the nonlinear problem when the nonlinearities involved  remain asymptotically  below the first Steklov-Neumann eigenvalue-line, which is a line connecting the first Steklov  and first Neumann  eigenvalues for $q$-Laplacian problems  with weights and a spectral parameter present  either in the differential equation or on the boundary.
We shall point out that the existence result for nonresonance problem has been studied in \cite{MNN, mavinga} for the case of Laplacian operator ($p=q=2$) and in \cite{ dGMR} for the case of  $p$-Laplacian  operator ($p=q\neq 2$). In this work, we extend these results to the case of $(p,q)$-Laplacian operator when the nonlinearities involved remain asymptotically below the first eigenvalue of  the $q$-Laplacian problem with a spectral parameter present in both the differential equation and on the boundary.

In  recent years, several authors have studied problems involving the  $(p,q)$-Laplacian operator due to its wide range of applications in physical and related sciences, see for example \cite{BDAFP, CI, PW}  and the references therein. Among  recent  studies concerning the $(p,q)$-Laplacian operator, we mention   \cite{SP,MSPKW}, where the authors dealt with  $(p,q)$-Laplacian problems with parameters in  bounded and unbounded $\mathbb{R}^ N$  domains with  nonlinear reaction terms in the interior and Dirichlet boundary conditions. Further recent studies on $(p,q)$-Laplacian problems with nonlinear reaction terms in the differential equation can be found in \cite{AEMV,SATEW, WBH} and the references therein.

To the best of our knowledge, the existence of nontrivial solutions to the problem \eqref{nonlinpde} has not been studied in the context where the nonlinearities in the differential equation and on the boundary interact, in some way, with the spectrum of the $(p,q)$-Laplacian problem \eqref{eig-prob_1}, where $p$ and $q$ are such that  $p,q>1$ and $p\neq q$. 

Throughout this paper, we will assume that $p,q\in (1,\infty)$ with $p\neq q$, and define $r:=\max\{p,q\}$. The functions  $c_i: \Omega \to \mathbb{R}$ with $i=1,2$,  $f:\overline{\Omega}\times \mathbb{R}\to \mathbb{R}$ and  $g:\overline{\Omega}\times \mathbb{R}\to \mathbb{R}$ satisfy the following conditions:
\begin{itemize}
    \item[(C1)] $c_1,c_2\in L^{\infty}(\Omega)$, $  c_1 \ge 0\, $, $ c_2\ge 0\, \;a.e.  \Omega,$ and\begin{equation*}\label{c:conditions}
   \int_\Omega c_1 >0 \mbox{ and } \int_\Omega c_2 >0.
\end{equation*} 
 \item[(C2)] $m, \rho\in L^{\infty}(\Omega)$, $  m \ge 0\,\, a.e. \Omega ,\,  \rho\ge 0\, \;a.e. \partial \Omega,$ and 
\begin{equation*}\label{mrho}
  \int_\Omega m+ \int_{\partial\Omega} \rho>0  .   
\end{equation*}

    \item[(C3)] $f\in C(\bar{\Omega}\times \mathbb{R})$ and $g\in C(\bar{\Omega}\times \mathbb{R})$.
    \item[(C4)] There exist constants $a_1, a_2>0$ such that 
\begin{equation*}
    |f(x,u)|\le a_1+a_2|u|^{s_1}, \qquad \mbox{a.e.  $x\in \Omega$ and for all $u\in \mathbb{R}$},
\end{equation*}
     with    $0\le s_1<r^*-1,$
where $r^*=\frac{rN}{N-r}$ if $r<N$ and  $r^*=\infty $ if $ r\ge N$.
     \item[(C5)] There exist constants $b_1, b_2>0$ such that 
     \begin{equation*}
      |g(x,u)|\le b_1+b_2|u|^{s_2}   , \qquad \mbox{a.e.  $x\in \Omega$ and for all $u\in \mathbb{R}$},
     \end{equation*}
    with  $0\le s_2<r_*-1, $
    where   $r_*=\frac{r(N-1)}{N-r}$ if $r<N$ and $r_*=\infty$ if $r\ge N$.
    \end{itemize}

\begin{definition}\label{def-sol-nonlin}
We say that  $u\in W:=W^{1, r}(\Omega)$,   is a weak solution of problem  \eqref{nonlinpde} if
\begin{align*}
    &\int_\Omega \left(|\nabla u|^{q-2}+|\nabla u|^{p-2}\right)\nabla u \cdot \nabla \varphi  +\int_\Omega \left(c_1(x)|u|^{p-2}+c_2(x)|u|^{q-2}\right) u \varphi  \nonumber\\
&=\int_\Omega f(x,u) \varphi  + \int_{\partial\Omega} g(x,u) \varphi , \ \ \text{for all}\ \ \varphi \in W. 
\end{align*}
\end{definition}

\begin{definition}\label{def_eig}
We say that $\lambda$ is eigenvalue of problem \eqref{eig-prob_1} if there exists $u_{\lambda}\in W \setminus \{0\}$ such that 
\begin{align}\label{def-eigen-1}
&\int_\Omega \left(|\nabla u_{\lambda}|^{q-2}+|\nabla u_{\lambda}|^{p-2}\right)\nabla u_{\lambda} \cdot \nabla w  +\int_\Omega \left(c_1(x)|u_{\lambda}|^{p-2}+c_2(x)|u_{\lambda}|^{q-2}\right) u_{\lambda} w \nonumber\\
&=\lambda \left( \int_\Omega m(x) |u_{\lambda}|^{q-2} u_{\lambda} w+\int_{\partial\Omega} \rho(x) |u_{\lambda}|^{q-2} u_{\lambda} w \right), \ \ \text{for all}\ \ w\in W. 
\end{align}
\end{definition}

Now,  define the following sets:
\begin{equation*}
\begin{aligned}
&\mathcal{C}:=\left\{ w\in W|  \int_\Omega m(x)|w|^{q-2}w+ \int_{\partial\Omega} \rho(x)|w|^{q-2}w=0 \right\},
\\
&\mathcal{C}_{\theta}:=\left\{ w\in W^{1,\theta}(\Omega)|  \int_\Omega m(x)|w|^{\theta-2}w+ \int_{\partial\Omega} \rho(x)|w|^{\theta-2}w=0 \right\}
\end{aligned}
\end{equation*}
and
\begin{equation}\label{def_C_1r} \mathcal{C}_{1,\theta }:=\mathcal{C}_\theta\cap \left\{ w\in W^{1,\theta}(\Omega)|  \int_\Omega m(x)|w|^\theta+ \int_{\partial\Omega} \rho(x)|w|^\theta=1 \right\},
\end{equation}
where $\theta \in \{p,\, q\}$. We also define
\begin{equation}\label{Lambda_q}
\Lambda_q:= \Lambda_q(m,\rho)=\inf_{w \in \mathcal{C} \setminus \{0\}} \frac{\|w\|_{c_2,q}^q}{\int_\Omega m(x)|w|^{q}+ \int_{\partial\Omega} \rho(x)|w|^{q}},
\end{equation}

\begin{equation}\label{Lambda_q2}
    \tilde{\Lambda}_q:=  \tilde{\Lambda}_q(m,\rho)= \inf_{w\in \mathcal{C} \setminus  \{0\}} \frac{\frac{1}{p}\|w\|_{c_1,p}^p+\frac{1}{q}\|w\|_{c_2,q}^q}{\frac{1}{q} \left(\int_\Omega m(x)|w|^{q}+ \int_{\partial\Omega} \rho(x)|w|^{q} \right)}
\end{equation}

and
\begin{equation}\label{eigen_q}
    \lambda_q:= \lambda_q(m,\rho)= \inf_{w\in \mathcal{C}_q \setminus  \{0\}}  \frac{\|w\|_{c_2,q}^q}{ \int_\Omega m(x)|w|^{q}+ \int_{\partial\Omega} \rho(x)|w|^{q} }.
\end{equation}

Since $\Omega$ is bounded, then $W\subseteq W^{1,p}(\Omega)$ and $W\subseteq W^{1,q}(\Omega)$, and therefore $\mathcal{C} \subseteq \mathcal{C}_q$ and $\mathcal{C} \subseteq \mathcal{C}_p$. This implies that $\lambda_q\le \Lambda_q$.
    For simplicity  notation, we denote
\begin{equation}\label{def-K_q}
K _q(w):=\int_\Omega m(x)|w|^{q}+ \int_{\partial\Omega} \rho(x)|w|^{q}, \quad w\in W^{1,q}(\Omega).
\end{equation}
 and observe that from (C2),  $K_q(w)\geqslant 0$, for all $w\in W^{1,q}(\Omega)$.

 In what follows, we give a complete description of the spectrum of the eigenvalue problem \ref{eig-prob_1}.
 
 \begin{theorem}\label{Neumman_theo}
    Assume that conditions (C1) and (C2) are satisfied. Then the spectrum of the eigenvalue problem \eqref{eig-prob_1} is $(\Lambda_q, \infty)$, where $\Lambda_q$ is given in \eqref{Lambda_q}. 
     \end{theorem}  

\begin{corollary}\label{coro-neu-stek}
Suppose that the assumptions of Theorem  \ref{Neumman_theo} hold. 
\begin{enumerate}
    \item[(i)] If $\rho\equiv 1$ and $m \equiv 0$ then the spectrum of the problem \eqref{eig-prob_1}, $(\Gamma_q, \infty)$, is the spectrum of the Steklov-like eigenvalue problem for the $(p,q)$-Laplacian,  where $\Gamma_q:= \Lambda_q(0,1)$.
    \item[(ii)] If $\rho \equiv 0$ and $m \equiv 1$ then the spectrum of the problem \eqref{eig-prob_1}, $(\Pi_q, \infty)$, is the spectrum of the Neumann-like eigenvalue problem for the $(p,q)$-Laplacian, where $\Pi_q:= \Lambda_q(1,0)$.
\end{enumerate}
\end{corollary}

 We shall mention that Part (i) of Corollary \ref{coro-neu-stek}  is also treated in \cite[Theorem 1]{BLMG}. Part (ii) of Corollary \ref{coro-neu-stek},  is discussed in \cite[Theorem 1.1]{MM} for the case $c_1\equiv c_2\equiv 0$. In this paper, we discuss the case where $m$ and $\rho $ are not necessarily identically zero, and $c_1\not\equiv 0$ and $c_2\not\equiv 0$. 
 
 Furthermore, we establish the following existence results for Problem \eqref{nonlinpde} when the nonresonance occurs.

 \begin{theorem}\label{main_theo}
    Assume that (C1)-(C5) hold. Let $F(x,u):=\int_0^uf(x,s)ds$  and $G(x,u):=\int_0^ug(x,s)ds$ 
 be such that there exist $\lambda,\mu \in \mathbb{R}$ which satisfy 
\begin{equation}\label{hypoth_FG}
        \limsup_{|u|\to \infty}\frac{qF(x,u)}{|u|^q}\leq \lambda m(x) \;\mbox{ and }\;   \limsup_{|u|\to \infty}\frac{qG(x,u)}{|u|^q}\leq \mu \rho(x)
    \end{equation}
    uniformly for $x\in \bar{\Omega}$ with $\max\{\lambda, \mu\}<\Lambda_q $.  \\
    Then problem \eqref{nonlinpde} has at least one weak solution $u\in W^{1,r}(\Omega)$ with $r=\max\{p,q\}$.
\end{theorem}

In the following theorem, we state the existence result when the nonlinearities involved asymptotically stay, in some sense, below the line connecting both the infimum of the set of eigenvalues of problem \eqref{eig-prob_1} with $m\equiv1$, $\rho \equiv0$, and the infimum of the set of eigenvalues of problem \eqref{eig-prob_1} with $m\equiv0$, $\rho \equiv 1$, as depicted in Figure \ref{recta1-clase1}.
\begin{theorem}\label{main_theo2}
    Assume that (C1)-(C5) are satisfied. Suppose that there exist $\lambda,\mu \in \mathbb{R}$ such that 
    \begin{equation}\label{hypoth_FG_2}
        \limsup_{|u|\to \infty}\frac{qF(x,u)}{|u|^q}\leq \lambda<\Pi_q \;\mbox{ and }\;   \limsup_{|u|\to \infty}\frac{qG(x,u)}{|u|^q}\leq \mu<\Gamma_q 
    \end{equation}
    uniformly for $x\in \bar{\Omega}$ with 
    \begin{equation}\label{hypoth_FG_3}
     \Pi_q \mu+\Gamma_q\lambda<\Pi_q \Gamma_q  , 
    \end{equation}
     where  $F(x,u):=\int_0^uf(x,s)ds$, $G(x,u):=\int_0^ug(x,s)ds$,  $\Pi_q:= \Lambda_q(1,0)$ and $\Gamma_q:= \Lambda_q(0,1)$.  \\
    Then problem \eqref{nonlinpde} has at least one weak solution $u\in W^{1,r}(\Omega)$ with $r=\max\{p,q\}$. 
\end{theorem}
Notice that the inequality \eqref{hypoth_FG} implies that, in some sense, some ratios of the nonlinearities involved
stay asymptotically in the plane-domain
\begin{equation*}
    R:=\{(\mu, \lambda)\in \mathbb{R}^2\,|\, \lambda <\Pi_q,\ \  \mu<\Gamma_q ,\ \  \Pi_q  \mu+\Gamma_q\lambda<\Pi_q \Gamma_q \},
\end{equation*}

which is a shaded region depicted in Figure \ref{recta1-clase1}.

\begin{figure}
    \centering
\begin{tikzpicture}
\draw[fill=gray!50!white]    (-2.7,2) -- ++(2.7,0) -- ++(2,-2) -- ++(0,-2.7)-- ++(-4.7,0);
\draw[black,very thick, ->] (0,-2.7) -- (0,5) node[right]{ $\mu$};
\draw[black,very thick, ->] (-2.7,0) -- (5,0) node[right]{ $\lambda$};
\draw[blue, very thick,dashed] (-2.5,4.5) -- (4.7,-2.7);
\draw[blue, very thick,dashed] (-2.7,2) -- (0,2);
\draw[blue,very  thick,dashed] (2,0) -- (2,-2.8);
\filldraw[black] (-0.2,-0.2) node[]{ $0$};
\filldraw[black] (-1.2,-1.2) node[]{ $R$};
\filldraw[black,very thick] (1.7,-0.3) node[]{ $\Pi_q$};
\filldraw[black,very thick] (-0.3,1.7) node[]{ $\Gamma_q$};
\filldraw[blue,very thick] (5,-1.2) node[]{\small{$\Pi_q \mu+\Gamma_q\lambda=\Pi_q \Gamma_q$}};
 \end{tikzpicture}
  \caption{Shaded region $R$} 
\end{figure}
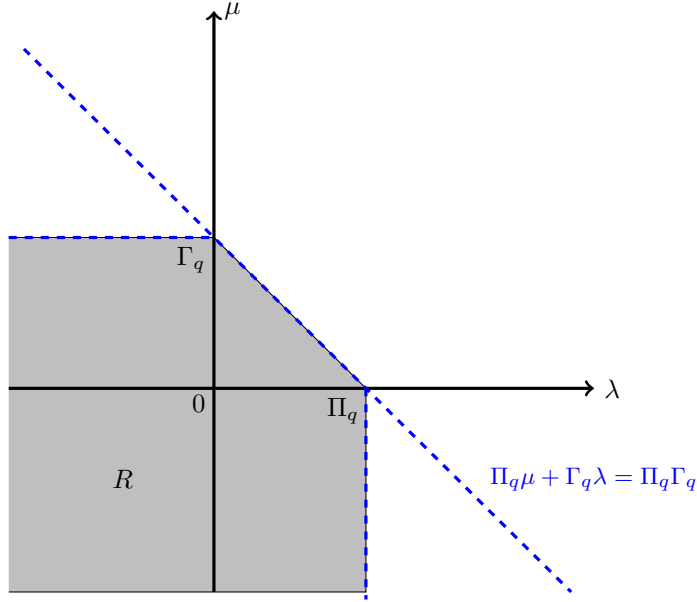\label{recta1-clase1}

The paper is organized as follows: Section 2 presents some preliminary and auxiliary results which will be needed in the proofs of the main results. Section 3 is devoted  to the description of the spectrum of the  problem \eqref{eig-prob_1}, where we prove Theorem \eqref{Neumman_theo}.
Finally, Section 4 deals with the nonlinear perturbation problems \eqref{nonlinpde} and the proofs of  Theorems \eqref{main_theo} and \eqref{main_theo2}.


\section{Preliminary and auxiliary results}

We will make use of the Sobolev space $W^{1,\theta} (\Omega)$ (with  $\theta \in (1,\infty)$), endowed with the  norm  \begin{equation*}
 \|u\|_{c_i,\theta}:=\left( \int_\Omega |\nabla u|^{\theta}+\int_\Omega c_i(x)|u|^{\theta}\right)^\frac{1}{\theta}. 
\end{equation*}

Because of the assumption {\rm (C1)}, we have that  $\|\cdot\|_{c_i,\theta}$, with $i=1,2$,  $1\leqslant \theta <\infty$, is equivalent to the usual norm of the Sobolev space $W^{1,\theta}(\Omega)$, $\displaystyle  \|u\|:= \left(\int_\Omega |\nabla u|^{\theta}+\int_\Omega |u|^{\theta}\right)^{\frac{1}{\theta}}$, see e.g. \cite[Lemma 1]{BLMG}. 

In addition to the Sobolev spaces, we use the Lebesgue spaces $L^\theta (\Omega)$ and $L^\theta (\partial \Omega)$, with norms 
\begin{equation*}
    \|u\|_\theta :=\left(\int_{\Omega} |u|^{\theta}\right)^\frac{1}{\theta} \qquad \mbox{ and } \qquad     \|u\|_{\partial \Omega, \theta}:=\left(\int_{\partial\Omega} |u|^{\theta} \right)^\frac{1}{\theta}
\end{equation*}
respectively. We shall consider  $L^\infty (\Omega)$ and $L^\infty (\partial\Omega)$ with the usual norms $\|\cdot \|_\infty$ and $\|\cdot \|_{\partial\Omega , \infty}$ respectively.

We also use the trace inequality  
\begin{equation}\label{trace_ineq}
    \displaystyle \|u\|_{\partial\Omega, \beta} \leqslant C_b \|u\|_{c_i,\theta},\ \ \text{for all}\ \ u\in W^{1,\theta}(\Omega),
\end{equation}  where $\beta\le \frac{(N-1)\theta}{N-\theta}$ if $1\le \theta <N$ and $\beta\ge 1$ if $\theta \ge N$, and $C_b$ is a positive constant (see e.g. \cite{KJF}, Theorem 6.4.1, page 316). Furthermore the embedding $W^{1,\theta}(\Omega) \hookrightarrow L^{\beta}(\Omega)$ is compact if $1<\theta<N$ and $1\leqslant \beta <\frac{(N-1)\theta}{N-\theta}$ (see \cite{KJF} section 6.10.5 page 344).

In this paper, we use  variational arguments. Therefore, in In the next propositions, we prove the continuous differentiability of the functionals 
$$I_1(u)=\frac{1}{p}\int_{\Omega} |\nabla u|^p \; \mbox{ and } \; I_2(u)=\frac{1}{p}\int_{E} c(x)|u|^p,$$
where $E$ is any measurable and bounded set in $\mathbb{R}^N$.

\begin{proposition}\label{prop-nabla}
Let $p>1$ and consider the functional $I_1:W^{1,p}(\Omega)\to \mathbb{R}$ defined by
\begin{equation*}
I_1(u)=\frac{1}{p}\int_{\Omega} |\nabla u|^p.
\end{equation*}
Then, $I_1\in C^1(W^{1,p}(\Omega))$ and 
 \begin{equation*}
I_1'(u)v=\int_{\Omega} |\nabla u|^{p-2} \nabla u\cdot\nabla v \qquad  \ \ \forall  u, v \in W^{1,p}(\Omega) \,.
\end{equation*}
\end{proposition}
\begin{proof}
    First, we prove that $I_1$ is Gateaux differentiable. We have to prove that for fixed $u,v\in W^{1,p}(\Omega)$
\begin{equation*}
   \lim_{t\to 0}\frac{1}{p}\int_{\Omega}\frac{|\nabla u+t\nabla v|^p-|\nabla u|^p}{t} 
   =\int_{\Omega} |\nabla u|^{p-2} \nabla u\cdot\nabla v  \,. 
\end{equation*}
Then we will have 
\begin{equation}\label{I_1_G}
    (I_1)'_G(u)v=\int_{\Omega} |\nabla u|^{p-2} \nabla u\cdot\nabla v   \,.
\end{equation}
In fact, given $x\in \Omega$ such that $\nabla u (x)$ and $\nabla v(x)$ are defined, and $0<|t|<1,$ by the mean value theorem there exists $|\theta(x)|\in [0,|t|]$ such that
\begin{align}\label{nabla_uv_1}
\left|\frac{|\nabla u(x)+t\nabla v(x)|^p-|\nabla u(x)|^p}{pt}\right|&=||\nabla u(x)+\theta(x)\nabla v(x)|^{p-2}(\nabla u(x)+\theta(x)\nabla v(x)) \nabla v(x)| \nonumber\\ &\leqslant |\nabla u(x)+\theta(x)\nabla v(x)|^{p-1}|\nabla v(x)|\nonumber
\\ &\leqslant C_p(|\nabla v(x)||\nabla u(x)|^{p-1}+|\nabla v(x)|^p)\,,
\end{align}
where $C_p>0$ is such that $(a+b)^{p-1}\leqslant C_p(a^{p-1}+b^{p-1})$ for all $a,b\geqslant 0$. 
Since $v\in W^{1,p}(\Omega)$, then $|\nabla v|\in L^p(\Omega)$, then by the Hölder inequality we have
\begin{equation}\label{nabla_uv_2}
\int_{\Omega}|\nabla v||\nabla u|^{p-1} \leqslant \|\nabla v\|_p\|\nabla u\|_p^{p-1}<\infty .
\end{equation}
Hence, from \eqref{nabla_uv_2} $|\nabla u+\theta\nabla v|^{p-1}|\nabla v|\in L^{1}(\Omega).$ From \eqref{nabla_uv_1} and using the Lebesgue dominated convergence theorem we have

    \begin{align*}
        \lim_{t\to 0}\frac{1}{p}\int_{\Omega}\frac{|\nabla u+t\nabla v|^p-|\nabla u|^p}{t}&= \int_{\Omega}\lim_{t\to 0}\frac{|\nabla u+t\nabla v|^p-|\nabla u|^p}{pt} \nonumber \\&=\int_{\Omega}\lim_{\theta\to 0}|\nabla u(x)+\theta(x)\nabla v(x)|^{p-2}(\nabla u(x)+\theta(x)\nabla v(x)) \nabla v(x) \nonumber \\&=\int_{\Omega} |\nabla u|^{p-2} \nabla u\cdot\nabla v  \,.
\end{align*}

This proves \eqref{I_1_G}.
We complete the proof by checking that the function $(I_1)'_G: W^{1,p}(\Omega)\to W^{-1,p'}(\Omega)$ is continuous. In fact, we take a sequence $(u_k)\subseteq W^{1,p}(\Omega)$ such that $u_k\to u$ in $W^{1,p}(\Omega)$, as $k\to \infty$. 
We can assume that, up to a subsequence, $|\nabla u_k|\to|\nabla u|$, $|\nabla u_k|\to |\nabla u|$ and $u_k\to u$ in $L^p(\Omega)$ as $k\to \infty$, and there exists $w, \hat{w}\in L^p(\Omega)$ such that $|\nabla u_k(x)|\leqslant w(x)$ and $ |u_k(x)|\leqslant \hat{w}(x)$ a.e $x\in \Omega$, for all $k$. Let $v\in W^{1,p}(\Omega)$ with $\|v\|_{1,p}\leqslant 1$. Then, by the Holder inequality
    \begin{align}\label{cont_gate_2}
        |((I_2)'_G(u_k)-(I_2)'_G(u))v| & \leqslant \int_\Omega ||\nabla u_k|^{p-2}\nabla u_k - |\nabla u|^{p-2}\nabla u||\nabla v| \nonumber \\
        & \leqslant \left( \int_\Omega ||\nabla u_k|^{p-2}\nabla u_k - |\nabla u|^{p-2}\nabla u|^{\frac{p-1}{p}}  \right)^{\frac{p}{p-1}}\left(\int_\Omega |\nabla v|^p \right)^\frac{1}{p} 
        \nonumber \\
        & \leqslant \left( \int_\Omega ||\nabla u_k|^{p-2}\nabla u_k - |\nabla u|^{p-2}\nabla u|^{\frac{p-1}{p}}  \right)^{\frac{p}{p-1}}.
        \end{align}
Now, for all $i=1, \dots, N$, $\partial_i u_k \to \partial_i u$ in $L^{p}(\Omega)$. Since there are a finite amount of $i$, then up to a subsequence, $\partial_i u_k(x) \to \partial_i u(x)$, a.e. $x\in \Omega$. Therefore, $\nabla u_k(x) \to \nabla u(x)$, a.e. $x\in \Omega$. Since the function $H(z):=|z|^{p-2}z$, $z\in \mathbb{R}^N$ is continuous, then \begin{equation}\label{nabla_uv_5}
    ||\nabla u_k(x)|^{p-2}\nabla u_k(x) - |\nabla u(x)|^{p-2}\nabla u(x)| \to 0, \quad \hbox{ as } k\to \infty , \quad \text{a.e. } x\in \Omega .
\end{equation}
On the other hand, a.e. $x\in \Omega$
    \begin{align}\label{nabla_uv_4}
        ||\nabla u_k(x)|^{p-2}\nabla u_k(x) - |\nabla u(x)|^{p-2}\nabla u(x)|^{\frac{p}{p-1}}|     \leqslant C'_p(|w(x)|^p+|\nabla u(x)|^p),
    \end{align}
where $C'_p>0$ satisfies $(a+b)^\frac{p}{p-1} \leqslant C'_p(a^\frac{p}{p-1}+b^\frac{p}{p-1}) $ for all $a,b\geqslant 0$.
Since $|\hat{w}(x)|^p+|\nabla u|^p \in L^1 (\Omega)$, then from \eqref{nabla_uv_4}, \eqref{nabla_uv_5} and by the Lebesgue dominated convergence theorem, 
\begin{equation}\label{nabla_uv_6}
    \lim_{k\to \infty} \int_\Omega ||\nabla u_k|^{p-2}\nabla u_k - |\nabla u|^{p-2}\nabla u|^{\frac{p-1}{p}}   =0.
     \end{equation}
Then, from  \eqref{cont_gate_2} and \eqref{nabla_uv_6}
\begin{equation*}
\lim_{k\to \infty}    \sup_{\|v\|\leqslant 1} |((I_1)'_G(u_k)-(I_1)'_G(u))v| =0.
\end{equation*}
This means that
\begin{equation*}
\lim_{k\to \infty}    \|(I_1)'_G(u_k)-(I_2)'_G(u)\| =0.
\end{equation*}
In consequence $(I_1)'_G$ is continuous. This completes the proof of the Proposition.
\end{proof}

\begin{proposition}\label{prop-cu}
Let $E$ a measurable set in $\mathbb{R}^N$, $c\in L^\infty(E)$ and $p>1$. Consider the functional $I_2:L^{p}(E)\to \mathbb{R}$ defined by
\begin{equation*}
I_2(u)=\frac{1}{p}\int_{E} c(x)|u|^p.
\end{equation*}
Then, $I_2\in C^1(L^{p}(E))$ and 
 \begin{equation*}
I_2'(u)v=\int_{E} c(x)|u|^{p-2} u\cdot v  \qquad \forall  u, v \in L^{p}(E) \,.
\end{equation*}
\end{proposition}
\begin{proof} 
The proof can be done similarly as the proof of Proposition \ref{prop-nabla}.
\end{proof}

\begin{proposition}
    \label{RemarkH}
For $\Lambda_q$ defined in \eqref{Lambda_q} we have the following characterization. 
\begin{equation}\label{Lambda_eta}
\Lambda _q =\inf_{w \in \mathcal{C} \setminus \eta } \frac{\|w\|_{c_2,q}^q}{K_q(w)} , 
\end{equation}
where $\eta:=\{w \in W \; |\; K_q(w)=0 \}$.
\end{proposition}
\begin{proof}
Let $A:=\left\{\frac{\|w\|_{c_2, q}^q}{K_q(w)}\; |\;  w\in \mathcal{C}\setminus \{0\}\right\}$ and $B:=\left\{ \frac{\|w\|_{c_2, q}^q}{K_q(w)}\; |\;  w\in \mathcal{C}\setminus \eta\right\}$. Since $\{0\} \subseteq \eta$, then $B\subseteq A$ and, in consequence, $\Lambda_q \leqslant \inf B$. Now, we shall prove that $\inf B \leqslant \inf A$. It suffices to show that $\inf B$ is a lower bound for $A$. 
 Indeed,  let $ w\in \mathcal{C}\setminus \{0\}$ such that $\frac{\|w\|_{c_2, q}^q}{K_q(w)} \in A$. If $w\in \mathcal{C}\setminus \eta$, then $\frac{\|w\|_{c_2, q}^q}{K_q(w)} \in B$. Therefore $\inf B \leqslant \frac{\|w\|_{c_2, q}^q}{K_q(w)}$. On the other hand, if $w\in \eta$ then $K_q(w)=0$. Furthermore, since $\|w\|_{c_2,q}^q>0$ then $\inf B\leqslant \infty=\frac{\|w\|_{c_2, q}^q}{K_q(w)}$. In any case  $\inf B \leqslant \frac{\|w\|_{c_2, q}^q}{K_q(w)}$, and so  $\inf B\leqslant \frac{\|w\|_{c_2, q}^q}{K_q(w)}$ for all $w\in \mathcal{C}\setminus \eta$. Thus, $\Lambda_q \ge \inf B$,  which completes the proof of \eqref{Lambda_eta}.
\end{proof}
 
    In a similar way,  we can establish the  following identities  using \eqref{Lambda_q2} and \eqref{eigen_q}

\begin{equation}\label{til_Lambda_eta}
\tilde \Lambda _q =\inf_{w \in \mathcal{C} \setminus \eta } \frac{\frac{1}{p}\|w\|_{c_1,p}^p+\frac{1}{q}\|w\|_{c_2,q}^q}{\frac{1}{q}K_q(w)} 
\end{equation}
and 
\begin{equation}\label{lambda_eta}
\lambda _q =\inf_{w \in \mathcal{C}_q \setminus \eta } \frac{\|w\|_{c_2,q}^q}{K_q(w)} . 
\end{equation}

\begin{proposition}\label{lamda_qpositive}
Let $\lambda_q$ as defined in \eqref{lambda_eta}. Then
\begin{equation}\label{lambda_eta_positive}
\lambda _q =\inf_{w \in \mathcal{C}_{1,q} \setminus \eta } \|w\|_{c_2,q}^q >0. 
\end{equation}
\end{proposition} 
\begin{proof}
    First, let us show that $\lambda _q =\inf_{w \in \mathcal{C}_{1,q} \setminus \eta } \|w\|_{c_2,q}^q.$ Let $A':=\{ \|w\|_{c_2, q}^q\; |\; w\in \mathcal{C}_{1,q}\setminus \eta\}$ and $B':=\left\{ \frac{\|w\|_{c_2, q}^q}{K_q(w)}\; |\;  w\in \mathcal{C}_q\setminus \eta\right\}$. Note that from \eqref{def_C_1r}, $A'\subseteq B'$. Then $\lambda_q=\inf B'\leqslant \inf A'=\inf_{w \in \mathcal{C}_{1,q} \setminus \eta } \|w\|_{c_2,q}^q.$ On the other hand, let $u\in \mathcal{C}_q\setminus \eta$. Then $w:=\frac{u}{\sqrt[q]{K_q(u)}}$  satisfies
\begin{align*}
    K_q(w)& = \frac{1}{K_q(u)} \left(\int_\Omega m(x)|u|^q+\int_{\partial\Omega}\rho |u|^q\right)=1.
\end{align*}
Therefore $w\in \mathcal{C}_{1,q}\setminus \eta$. 
Since $\|\cdot \|_{c_2, q}$ is a norm, then $\frac{\|u\|_{c_2,q}^q}{K_q(u)}=\|w\|_{c_2,q}^q\in A'$. Hence $B'\subseteq A'$ and $\lambda_q=\inf B'\geqslant \inf A'=\inf_{w \in \mathcal{C}_{1,q} \setminus \eta } \|w\|_{c_2,q}^q.$ Thus,  $\lambda _q =\inf_{w \in \mathcal{C}_{1,q} \setminus \eta } \|w\|_{c_2,q}^q $.
\par Observe that  $\lambda_q\geqslant 0.$  Now, let us show that $\lambda_q\ne 0$. Assume by contradiction that $\lambda_q=0$. By the characterization  of infimum in \eqref{lambda_eta_positive}, there exists a minimizing sequence $(w_n)\subseteq \mathcal{C}_{1,q} \setminus \eta$, that is, 
\begin{equation*}\label{limi}
    \lim_{n\to \infty}\|w_n\|_{c_2,q}^q=\lambda_q=0.
\end{equation*}
and $K_q(w_n)=1, \forall n\geqslant 1.$ Therefore, $w_n\to 0$ in $W^{1,q}(\Omega)$. By the continuous  embedding $W^{1,q}(\Omega) $ in $L^q(\Omega)$ and the trace inequality \eqref{trace_ineq}, $w_n\to 0$ in $L^{q}(\Omega)$ and $w_n\to 0$ in $L^{q}(\partial\Omega)$. 
From Proposition \ref{prop-cu} and the fact that $m\in L^\infty (\Omega)$ and $\rho \in L^\infty (\partial\Omega)$, we get 
\begin{equation*}
    \lim_{n\to \infty}      \int_\Omega m(x)|w_{n}|^q=0 \quad
\mbox{ and } \quad
 \lim_{n\to \infty}      \int_{\partial\Omega} \rho(x)|w_{n}|^q=0.
\end{equation*}
Hence,
\begin{equation*}
    \lim_{n\to \infty}K_q(w_{n})=0.
\end{equation*}
This leads a contradiction because $K_q(w_n)=1, \forall n\geqslant 1.$ Thus $\lambda_q>0.$
\end{proof}

\begin{proposition}\label{remark-lambda>0}
If $\lambda$ is an eigenvalue of problem \eqref{eig-prob_1} and $u_\lambda$ is a corresponding eigenfunction, then $\lambda>0$ and   $\displaystyle\int_\Omega m(x) |u_\lambda|^q + \int_{\partial\Omega} \rho(x) |u_\lambda|^q >0$.
\end{proposition}
\begin{proof}
Taking $w=u_\lambda$ as a test function in \eqref{def-eigen-1}, we have
\begin{equation}\label{eigenfunction}
\|u_\lambda\|_{c_1,p}^p+    \|u_\lambda\|_{c_2,q}^q=\lambda \left( \int_\Omega m(x) |u_\lambda|^q + \int_{\partial\Omega} \rho(x) |u_\lambda|^q \right).
\end{equation}
Since $u_\lambda \neq 0$ as an eigenfunction, the left hand side of \eqref{eigenfunction} is positive. Thus, $\lambda >0$ and  $\displaystyle\int_\Omega m(x) |u_\lambda|^q + \int_{\partial\Omega} \rho(x) |u_\lambda|^q >0$.
\end{proof}

\section{The spectrum  problems for $(p,q)-$Laplacian}
In this section we will give a complete description of the spectrum of problem \eqref{eig-prob_1}  when $c_1$, $ c_2,$ and the weights $ m$ and $\rho$ satisfy conditions (C1) and (C2).

\begin{lem}\label{0<lambda_q}
suppose that conditions (C1) and (C2) are satisfied. Then 
\begin{enumerate}
    \item[(i)]  $0<\lambda_q\leqslant \Lambda_q$.
\item[(ii)]  $\Lambda_q =\Tilde{\Lambda}_q$. 
\item[(iii)] there are no eigenvalues of problem \eqref{eig-prob_1} on $(-\infty, \Lambda_q]$,    
\end{enumerate}
where $\Lambda_q$, $\tilde{\Lambda}_q$ and $\lambda_q$ are defined in \eqref{Lambda_q}, \eqref{Lambda_q2} and \eqref{eigen_q} respectively.
\end{lem}
\begin{proof}
\begin{enumerate}
    \item[(i)]  By Proposition \ref{lamda_qpositive} $\lambda_q >0$.  On the other hand, since $\mathcal{C}\subseteq \mathcal{C}_q$ then $\Lambda_q \geqslant \lambda_q$, which also yields to $\Lambda_q >0$.
\item[(ii)] Let us show that $\Lambda_q=\Tilde{\Lambda}_q.$ From the definitions of $\Lambda_q$ and $\Tilde{\Lambda}_q$ in \eqref{Lambda_q} and \eqref{Lambda_q2}, we clearly have  that  $\Lambda_q\leqslant \Tilde{\Lambda}_q$. Therefore, it is enough to prove that $\Lambda_q\geqslant\Tilde{\Lambda}_q$. Let $v\in \mathcal{C} \setminus \{0\}$ and $t>0$. Then $tv\in \mathcal{C}\setminus \{0\}$. Hence,
\begin{equation}\label{lambda_hat_1}
\begin{aligned}
    \tilde{\Lambda}_q \leqslant \frac{\dfrac{t^q}{q}\|v\|_{c_2,q}^q+\dfrac{t^p}{p}\|v\|_{c_1,p}^p}{\dfrac{t^q}{q} K_q(v)}   = \dfrac{\|v\|_{c_2,q}^q }{K_q(v)} + \dfrac{qt^{p-q}}{p}\frac{\|v\|_{c_1,p}^p}{K_q(v)}.
   \end{aligned}
\end{equation}
If $q>p$ and taking limit in \eqref{lambda_hat_1} as $t\to \infty$, we have that $\dfrac{qt^{p-q}}{p}\frac{\|v\|_{c_1,p}^p}{K_q(v)}\to 0$. If $q<p$ and taking the limit as $t\to 0^+$, we also have $\dfrac{qt^{p-q}}{p}\frac{\|v\|_{c_1,p}^p}{K_q(v)}\to 0$. In either case we have that for any $v\in W\setminus \{0\}$
\begin{equation}\label{destilde}
\tilde{\Lambda}_q \leqslant \frac{\|v\|_{c_2,q}^q}{K_q(v)}.
\end{equation}
Taking the infimum in \eqref{destilde}, we have  $\tilde{\Lambda}_q \leqslant \Lambda_q$. Thus $\tilde{\Lambda}_q = \Lambda_q$.
\item[(iii)]
In view of Proposition \ref{remark-lambda>0}, any $\lambda\leqslant 0$ is not an eigenvalue of \eqref{eig-prob_1}. 
Now, we show that if $\lambda \in (0, \Lambda_q]$ then $\lambda$ is not an eigenvalue of \eqref{eig-prob_1}.  Arguing by contradiction let us assume that there exists $0<\lambda \leqslant \Lambda_q$ such that $\lambda$ is an eigenfunction of \eqref{eig-prob_1}. Then there exists an eigenfunction $u_\lambda$, which in particular satisfies \eqref{def-eigen-1} with $w=u_\lambda$, producing
\begin{equation*}\label{eigen-3}
\|u_\lambda\|_{c_1,p}^p+\|u_\lambda\|_{c_2,q}^q=\lambda K_q(u_\lambda).
 \end{equation*}
 Therefore 
\begin{equation}\label{eigen-4}
\frac{\|u_{\lambda}\|_{c_2,q}^q+\|u_{\lambda}\|_{c_1,p}^p}{K_q(u_\lambda)} =\lambda \leqslant  
\Lambda_q\leqslant \frac{\|u_\lambda \|_{c_2,q}^q}{K_q(u_\lambda )}. 
\end{equation}
From \eqref{eigen-4}, we get that $ \|u_{\lambda}\|_{c_1,p}^p\leqslant 0.$
Hence, $\|u_{\lambda}\|_{c_1,p}^p=0$ which implies that $u_{\lambda}=0$ a.e. in $\Omega$.  This leads to a contradiction to the fact that $u_{\lambda}$ is an eigenfunction of \eqref{eig-prob_1}. This completes the proof.
\end{enumerate}
\end{proof}

Now, consider the the functional $J_\lambda : W\to \mathbb{R}$ defined by
\begin{equation}\label{def_J_lamb}
    J_\lambda (u)=\frac{1}{p}\|u\|_{c_1,p}^p+\frac{1}{q}\|u\|_{c_2,q}^q-\frac{\lambda}{q}\left(\int_{\Omega} m(x) |u|^q+\int_{\partial\Omega} \rho(x)|u|^q\right).
\end{equation}
From Propositions \ref{prop-nabla} and \ref{prop-cu} we have that for any $\lambda$, $J_\lambda$ is a $C^1$ functional and its derivative is given by
\begin{eqnarray*}\label{deriv_J}
    \langle J'_{\lambda} u, w \rangle &=&\int_\Omega \left(|\nabla u|^{p-2}\nabla u \cdot \nabla w + c_1(x)|u|^{p-2}u w + |\nabla u|^{q-2}\nabla u \cdot \nabla w +c_2(x)|u|^{q-2}u w \right) \nonumber \\
    &-& \lambda \left(\int_{\Omega} m(x) |u|^{q-2}uw+\int_{\partial\Omega} \rho(x)|u|^{q-2}uw\right) ,
\end{eqnarray*}
for all $w\in W$.
Therefore, according to the Definition \ref{def_eig}, the eigenvalues of problem \eqref{eig-prob_1}
 are precisely the values of  $\lambda$ such that $J_\lambda$ has at least one nontrivial critical point. \\
Observe that the trace inequality \eqref{trace_ineq} and the definition of $K_q$ \eqref{def-K_q} give the following estimate 
    \begin{equation}\label{bound-K_q}
        K_q(w) \leqslant \|m\|_\infty \|w\|_q^q+ \|\rho \|_\infty \|w\|_{\partial\Omega, q}^q \leqslant C_1 \|w\|_q^q, \mbox{ for all } w\in W,
    \end{equation}
    where $C_1:=\|m\|_\infty+ \|\rho \|_\infty C_b^q$.  \begin{lem}\label{J_lambda-coer}
 If $q<p$, then
 for any $\lambda>0$, $J_\lambda$ is coercive.
\end{lem}
\begin{proof}
    Observe that $W=W^{1,p}(\Omega)$.
    Arguing by contradiction,  suppose that there exist a constant $C>0$ and a sequence $(u_n)$ in $W$ such that $\|u_n\|_{c_1,p} \to \infty $ as $n\to \infty$ and $J_\lambda (u_n)\leqslant C$. Then,  from \eqref{def_J_lamb} and \eqref{bound-K_q} we see that
    \begin{align}\label{coer_1}
 0  \leqslant \frac{1}{p}\|u_n\|_{c_1,p}^p    \leqslant -\frac{1}{q}\|u_n\|_{c_2,q}^q+ \frac{\lambda}{q}K_q(u_n) +C 
 \leqslant  \frac{\lambda C_1}{q} \|u_n\|_q^q+C. 
    \end{align}
Since $\|u_n\|_{c_1,p}\to \infty$, then from \eqref{coer_1} $\|u_n\|_{q}\to \infty$. Set $v_n:=\frac{u_n}{\|u_n\|_q}$. Dividing  \eqref{coer_1} by $\|u_n\|_q^p$ we have
   \begin{equation}\label{coer_2}
 0\leqslant \frac{1}{p}\|v_n\|_{c_1,p}^p  \leqslant  \frac{\lambda C_1}{q\|u_n\|_q^{p-q}} +\frac{C}{\|u_n\|_q^{p}}. 
    \end{equation}
Noting $p-q>0$, then the right hand side of  \eqref{coer_2} goes to zero as $n\to \infty$. Consequently $\lim_{n\to \infty}\|v_n\|_{c_1,p}=0$ and hence $v_n\to 0$ in $W^{1,p}(\Omega)$.
Now,  using the continuous embedding of $L^p(\Omega) \hookrightarrow L^q(\Omega)$ and $W^{1,p}(\Omega) \hookrightarrow L^p(\Omega)$, we have
\begin{equation*}
    1=\|v_n\|_q \leqslant C_1\|v_n\|_p \leqslant C_2 \|v_n\|_{c_1,p}\to 0,
\end{equation*}
which is a contraction. Thus $J_{\lambda}$ is coercive for  $q<p$.
\end{proof}

\begin{lem}\label{1-exist-eigen}
    If $q<p$, then every $\lambda >\Lambda_q$ is an eigenvalue of the problem \eqref{eig-prob_1}.
\end{lem}
\begin{proof}
We first prove that for any $\lambda > \Lambda_q$, $J_\lambda$ is weakly lower semicontinuous. Let $(u_n)$ be a sequence such that $u_n \rightharpoonup u$ in $W$.
We want to see that $  \displaystyle{  J_\lambda (u) \leqslant \liminf_{n\to \infty} J_\lambda (u_n)}$. 
Arguing by contradiction, let us assume that 
\begin{equation}\label{liminf<norm-1}
 \liminf_{n\to \infty} \left(\frac{1}p \|u_n\|_{c_1,p}^p + \frac{1}q \|u_n\|_{c_2,q}^q - \frac{\lambda}q K_{q}(u_n) \right) <   \frac{1}p \|u\|_{c_1,p}^p+\frac{1}q \|u\|_{c_2,q}^q- \frac{\lambda}q K_{q}(u).
\end{equation}
Since $q<p$, then  $W\subseteq W^{1,q}(\Omega)$ and therefore, $u_n \rightharpoonup u$ in $W^{1,q}(\Omega)$. Using the fact that the norms in Banach spaces  are weakly lower semicontinuous and properties of $\liminf$, we get
\begin{eqnarray}\label{norm_leq_liminf}
   \frac{1}p \|u\|_{c_1,p}^p+\frac{1}q \|u\|_{c_2,q}^q 
   &\leqslant & \liminf_{n\to \infty} \left(\frac{1}p \|u_n\|_{c_1,p}^p + \frac{1}q \|u_n\|_{c_2,q}^q \right).
\end{eqnarray}
From \eqref{liminf<norm-1} and the characterization of $\liminf$ as the infimum of the set of subsequential limits, we have that for some subsequence 
\begin{equation}\label{liminf<norm}
 \lim_{k\to \infty} \left(\frac{1}p \|u_{n_k}\|_{c_1,p}^p + \frac{1}q \|u_{n_k}\|_{c_2,q}^q - \frac{\lambda}q K_{q}(u_{n_k}) \right) <   \frac{1}p \|u\|_{c_1,p}^p+\frac{1}q \|u\|_{c_2,q}^q- \frac{\lambda}q K_{q}(u) .
\end{equation}
Since $u_{n_k} \rightharpoonup u$ in $W^{1,q}(\Omega)$ then $(u_{n_k})$ is bounded in $W^{1,q}(\Omega)$. By the compact embedding of $W^{1,q}(\Omega)$ in  $L^q(\Omega)$ we have that there is a  subsequence  (which without loss of generality we relabel $(u_{n_k})$) such  that $u_{n_k}\to u$ in $L^q(\Omega)$. 
Furthermore, from \eqref{trace_ineq} we get $u_{n_k}\to u$ in $L^{q}(\partial\Omega)$. By proposition \ref{prop-cu}  with $E=\Omega$ or $E=\partial\Omega$ and $c(x)=m(x)$ or $c(x)=\rho(x)$, we have that $K_q$ is continuous. Hence,

\begin{equation}\label{K_q-u_n}
    \lim_{k\to \infty}K_q(u_{n_k})=K_q(u).
\end{equation}
From \eqref{liminf<norm}  and \eqref{K_q-u_n} 
\begin{equation*}
\lim_{k\to \infty} \left(\frac{1}p \|u_{n_k}\|_{c_1,p}^p + \frac{1}q \|u_{n_k}\|_{c_2,q}^q \right) <   \frac{1}p \|u\|_{c_1,p}^p+\frac{1}q \|u\|_{c_2,q}^q,
\end{equation*}
which contradicts \eqref{norm_leq_liminf}. Therefore $J_\lambda$ is weakly lower semicontinuous. 

Now, we will prove that any $\lambda > \Lambda_q$ is an eigenvalue of Problem \eqref{eig-prob_1}. Since $\lambda > \Lambda_q$, then by assertion (i) of  Lemma \ref{0<lambda_q}, $\lambda>0$. It follows from Lemma \ref{J_lambda-coer} that  $J_\lambda$ is coercive. By \cite[Theorem 1.6 and Remark 1.5.7]{Badiale-Serra_2010}, $J_\lambda $ has a global minimizer $w_* \in W$, that is, 
\begin{equation*}
    J_\lambda (w_*) =\min_{w\in W} J_\lambda (w).
\end{equation*}
Now let us prove that $w_*\neq 0$. Let $\lambda >\tilde{\Lambda}_q$, where $\tilde{\Lambda}_q$ is as  defined in \eqref{Lambda_q}. Then, from \eqref{til_Lambda_eta}, there exists $w_\lambda \neq 0$ such that 
\begin{equation*}
    \frac{\frac{1}{p}\|w_\lambda \|_{c_1,p}^p+\frac{1}{q}\|w_\lambda \|_{c_2,q}^q}{\frac{1}{q}K_q(w_\lambda)}<\lambda .
\end{equation*}
Hence,
\begin{equation*}
    \frac{1}{p}\|w_\lambda \|_{c_1,p}^p+\frac{1}{q}\|w_\lambda \|_{c_2,q}^q< \frac{\lambda}{q}K_q(w_\lambda).
\end{equation*}
This implies that
\begin{equation*}
 J_\lambda (w_*)\leqslant   J_\lambda (w_\lambda) <0.
\end{equation*}
So, $w_* \neq 0$. Thus $\lambda$ is an eigenvalue of \eqref{eig-prob_1}.
\end{proof}

In our approach to the statement of the coercivity and weakly lower semicontinuity of $J_\lambda$ the condition  $q<p$ was essential. However, when $q>p$, this procedure does not apply. To overcome this situation, we shall consider the restriction of $J_{\lambda}$ to the Nehari manifold  defined below and prove that for any $\lambda >\Lambda_q$, this functional attains its minimum.

Assume that  $q>p$.  Then $W=W^{1,q}(\Omega)$.
Consider the Nehari manifold defined by
\begin{align}\label{def_Nehari}\mathcal{N}_{\lambda}&=\{v\in W\backslash\{0\}\;|\;\langle J'_\lambda (v),v\rangle=0\}\nonumber\\
    &=\{v\in W\backslash\{0\}\;|\; \|v \|_{c_1,p}^p+\|v \|_{c_2,q}^q=\lambda K_q(v)\}.
\end{align}
Define $J_{\lambda}:  \mathcal{N}_{\lambda}\to (0,\infty) $.
 Observe that if $v\in \mathcal{N}_{\lambda}$ then from \eqref{def_Nehari}  
\begin{equation}\label{J_rest_eta}
    J_{\lambda}(v)=\frac{q-p}{pq}\|v \|_{c_1,p}^p.
\end{equation}
\begin{lem}\label{2-exist-eigen}
     If $q>p$ and $\lambda>\Lambda_q$, then
     \begin{itemize}
         \item[(i)] $ \mathcal{N}_{\lambda}\neq \emptyset.$
         \item[(ii)] Every minimizing sequence  for $J_{\lambda}$ restricted to $\mathcal{N}_{\lambda}$ is bounded in $W$, that is, if $\left(u_{n}\right) \subseteq \mathcal{N}_{\lambda}$ satisfies 
\begin{equation}\label{Jlambda}
\lim_{n\to \infty} J_\lambda (u_n) = m_{\lambda} :=\inf _{w \in \mathcal{N}_{\lambda}} J_{\lambda}(w), 
\end{equation}
then $(u_n)$ is bounded in $W$.
         \item[(iii)]  $m_{\lambda}>0$.
         \item[(iv)] 
     There exists $u_{*} \in \mathcal{N}_{\lambda}$ where ${J}_{\lambda}$ attains its minimal value, $\displaystyle {m_{\lambda}=J_{\lambda}(u_*)}$.
     \end{itemize}
     
\end{lem}
\begin{proof}

\begin{itemize}
    \item[(i)] We will show that $\mathcal{N}_{\lambda}\ne \emptyset$. Indeed,
since $\lambda>\Lambda_{q}$, it follows from  \eqref{Lambda_q}  that there exists $v_{0} \in W \backslash\{0\}$ such that $\|v_0\|_{c_2,q}^q<\lambda K_q(v_{0})$. 
We claim that there exists $\tau>0$, such that $\tau v_{0} \in \mathcal{N}_{\lambda}$.  Observe that  $\tau v_{0} \in \mathcal{N}_{\lambda}$ means  that 
\begin{equation}\label{tau_v_1}
    \tau^{p} \|v_0\|_{c_1,p}^p+\tau^{q} \|v_0 \|_{c_2,q}^q=\lambda \tau^{q}K_q(v_{0}).
\end{equation}
Dividing  by $\tau^{p},$ equation \eqref{tau_v_1} is equivalent to 
\begin{equation}\label{tau_v_2}
    \|v_0 \|_{c_1,p}^p+\tau^{q-p} \|v_0 \|_{c_2,q}^q=\lambda \tau^{q-p}K_q(v_{0}).\end{equation}

Taking into account that $\lambda K_q(v_{0})-\|v_0\|_{c_2,q}^q>0$, we see that equation \eqref{tau_v_2} can be solved for $\tau$, with 
\begin{equation}\label{tau_v_3}
\tau=\left(\frac{\|v_0\|_{c_1,p}^p}{\lambda K_q(v_{0})-\|v_0\|_{c_2,q}^q}\right)^{\frac{1}{q-p}}
,
\end{equation}
which is positive as $\|v_{0}\|_{c_1,p}>0$. Therefore, choosing  $\tau$ as in \eqref{tau_v_3}, we have $\tau v_{0} \in \mathcal{N}_{\lambda}$.
\item[(ii)]  Let 
$\left(u_{n}\right) \subseteq \mathcal{N}_{\lambda}$ be a minimizing sequence for $J_{\lambda}$. We claim that $(u_{n}) $  is bounded in $W$. Assume to the contrary that $\left(u_{n}\right)$ is unbounded in $W$. Then there exists a subsequence, again denoted $\left(u_{n}\right)$, such that $\left\|u_{n}\right\|_{c_2,q} \rightarrow \infty$. 
Define $v_{n}=u_{n} /\left\|u_{n}\right\|_{c_2,q}.$ Then $\left\|v_{n}\right\|_{c_2,q}=1. $ Since $W$ is reflexive,  one has (by going to a subsequence relabeled $(v_n)$, if need be) that there is $v_{0} \in W$ such that $v_{n} \rightharpoonup v_{0}$ in $W$. Due to the compact embedding of $W^{1,q}(\Omega)$ in $ L^q(\Omega)$ and the continuous embedding of  $L^q(\Omega)$ into $L^p(\Omega)$, (for a subsequence similarly relabeled) $v_{n} \rightharpoonup v_{0}$ in $W^{1, p}(\Omega)$ and  $v_{n} \rightarrow v_{0}$ in $L^{q}(\Omega)$  and $v_{n} \rightarrow v_{0}$ in $L^{p}(\Omega)$.  From \eqref{J_rest_eta} and \eqref{Jlambda}, we have
\begin{equation}\label{Nehari_0}
    \lim_{n\to \infty}\frac{q-p}{q p} \|u_n\|_{c_1,p}^p= m_{\lambda}.
\end{equation}
Since $\displaystyle \lim_{n\to \infty} \|u_n\|_{c_2, q}=\infty$, then from \eqref{Nehari_0} we get 
\begin{equation*}
\lim_{n\to \infty}\frac{q-p}{q p} \|v_n\|_{c_1,p}^p=\lim_{n\to \infty}\frac{\frac{q-p}{q p} \|u_n\|_{c_1,p}^p} {\left\|u_{n}\right\|_{c_2,q}^p} = 0.
\end{equation*}
This implies $ v_n\to 0$ in $W^{1,p}(\Omega)$ as $n\to \infty$. By the continuous embedding of $W^{1,p}(\Omega)$ in $L^p(\Omega)$, it follows that $v_n \to 0$ in $L^p(\Omega)$ as $n\to \infty$. Hence, $v_0(x)=0$ a.e. $x\in \Omega$, which implies $\|v_{n}\|_{q}\to 0$  as $n\to \infty$.
Using the fact that  $u_n \in \mathcal{N}_\lambda$, we have that  
\begin{equation}\label{Nehari_1}0 \leqslant \|u_n \|_{c_1,p}^p=\lambda K_q(u_{n})-\|u_n \|_{c_2,q}^q , \qquad \forall n  \geqslant 1.  \end{equation}
Dividing \eqref{Nehari_1} by $\|u_n\|_{c_2,q}^q$ and taking into account  \eqref{bound-K_q}, we get
\begin{equation}\label{vn}
0 \leqslant \frac{\|u_n \|_{c_1,p}^p}{\|u_n\|_{c_2,q}^q}=\lambda C_1 \|v_{n}\|_{q}^{q}-1 ,\ \ \forall n \geqslant 1 .
\end{equation}
Taking the limit as  $n\to\infty$ in \eqref{vn}  we get $0\leqslant -1.$ This is a contradiction. Thus, $\left(u_{n}\right)$ is bounded in $W$.

\item[(iii)] 
Let us show that $m_{\lambda} >0$. Assume by contradiction that $m_{\lambda}=0$ and let $\left(u_{n}\right) \subseteq \mathcal{N}_{\lambda}$ be a minimizing sequence for $J_{\lambda}$. From $\mbox{(ii)}$  $\left(u_{n}\right)$ is bounded in $W$. 
Since $W$ is reflexive, there exists a subsequence, relabeled again $(u_n)$, such that  $u_{n} \rightharpoonup u_{0}$  in $W$, $u_n\rightharpoonup u_{0}$ in $W^{1, p}(\Omega)$ and $u_{n} \rightarrow u_{0}$ in both $L^{q}(\Omega) $ and  $L^{p}(\Omega)$ because of  the compact embedding of $W^{1,q}(\Omega)$ in $ L^q(\Omega)$ and the continuous embedding of  $L^q(\Omega)$ into $L^p(\Omega)$.
From \eqref{J_rest_eta}, \eqref{Jlambda}  and the fact that $m_\lambda=0$, we have $\|u_n \|_{c_1,p}^p\rightarrow 0.$ Therefore $u_n\to 0 \mbox{ in } L^p({\Omega})$. Hence,  by the uniqueness of limit,  $u_{0} \equiv 0.$
Define $w_{n}=u_{n} /\left\|u_{n}\right\|_{ q}, n \geq 1$.  Observe  $(w_n)$ is bounded in $L^q(\Omega)$ and $\|w_n\|_q=1$ for all $n$. We assert that $\left(w_{n}\right)$ is also bounded in $W$.
Indeed, since $\left(u_{n}\right) \subseteq \mathcal{N}_{\lambda}$, then it follows from \eqref{def_Nehari}  that   $u_n\neq 0$ for each $n$ and $\|u_n\|_{c_1,p}^p+\|u_n \|_{c_2,q}^q=\lambda K_q(u_n)$.  Dividing by $\|u_n\|_q^q$ and taking into account \eqref{bound-K_q}, we obtain  
\begin{equation}\label{w_necu}
\frac{\|w_n \|_{c_1,p}^p}{\left\|u_{n}\right\|_{q}^{q-p}}+\|w_n \|_{c_2,q}^q=\lambda \frac{K_q(u_n)}{\|u_n\|_q^q} \leqslant \lambda C_1.
\end{equation}
Therefore, 
\begin{equation*}\label{w_bound}
    \|w_n \|_{c_2,q}^q\leqslant\lambda C_1,
\end{equation*}
which implies that $\left(w_{n}\right)$ is bounded in $W$. By the reflexivity of $W$, there exists a subsequence, relabeled again $(w_n)$, such that  $w_{n} \rightharpoonup w_{0}$  in $W$, $w_{n} \rightharpoonup w_{0}$ in $W^{1, p}(\Omega)$  and $w_{n} \rightarrow w_{0}$ in $L^{q}(\Omega).$   
From \eqref{w_necu}, we have 
\begin{equation}\label{w_eq}
    \|w_n \|_{c_1,p}^p\leqslant \left\|u_{n}\right\|_{q}^{q-p}(\lambda C_1- \|w_n \|_{c_2,q}^q).
\end{equation}

Using the fact $ (w_n )$ is bounded in $W$,  $u_{n} \rightarrow 0$ in $L^{q} (\Omega)$ and $q-p>0$, we have that \eqref{w_eq} leads to  $\|w_n \|_{c_1,p}\rightarrow 0.$ Then $w_n\rightarrow 0$ in $W^{1,p}(\Omega).$  By the uniqueness of limit, $w_0=0$. Therefore, $w_{n} \rightarrow 0$ in $L^{q}(\Omega)$. Hence, $\|w_n\|_q\to 0$ and $n\to \infty$.  This contradicts the fact that $\|w_n\|_q=1$ for all $n$.
\vspace{0.3cm}

\item[(iv)] Let us show that there exists $u_{*} \in \mathcal{N}_{\lambda}$ such that $J_{\lambda}\left(u_{*}\right)=m_{\lambda}$.\\
Let $\left(u_{n}\right) \subseteq \mathcal{N}_{\lambda}$ be a minimizing sequence, that is, $J_{\lambda}\left(u_{n}\right)=\frac{q-p}{pq}\|u_n \|_{c_1,p}^p \rightarrow m_{\lambda}$ as $n\to \infty$. Then \begin{equation}\label{K_p}
    \lim_{n\to\infty}\|u_n \|_{c_1,p}^p=\frac{qp}{q-p}m_{\lambda}.
\end{equation}
By $\mbox{(ii)}$ $\left(u_{n}\right)$ is bounded in $W$,  and therefore there exists a subsequence $\left(u_{n}\right)$ such that $u_{n} \rightharpoonup u_{*} \mbox{ in } W$ and so, $u_{n} \rightharpoonup u_{*} \in W^{1,p}(\Omega)$ and $u_{n} \to u_{*}$ in $L^{q}(\Omega)$ and $L^{p}( \Omega)$.

It follows from the properties of weak convergence  that
\begin{equation}\label{kp}
    \|u_* \|_{c_1,p}^p \leqslant \liminf_{n\to \infty} \|u_n \|_{c_1,p}^p
\end{equation}
and 
\begin{equation}\label{kq}
    \|u_*\|_{c_2,q}^q\leqslant \liminf_{n\to \infty} \|u_n \|_{c_2,q}^q.
\end{equation}
Since $(u_n)\subseteq \mathcal{N}_{\lambda}$, from the continuity of $K_q$ (see Proposition \ref{prop-cu}) it follows that 
\begin{align*}
    \lambda K_q(u_{*})&= \lambda \lim_{n\to\infty} K_q(u_n)=  \liminf_{n\to \infty}( \|u_n \|_{c_1,p}^p+\|u_n \|_{c_2,q}^q).
    \end{align*}

By \eqref{kp} and \eqref{kq} and the properties of the liminf, we have 

\begin{equation}\label{deskpkq}
    \|u_*\|_{c_1,p}^p+\|u_* \|_{c_2,q}^q\leqslant  \lambda K_q(u_{*}) .
\end{equation}
We assert that $u_{*}\neq 0.$ Assume to the contrary that $u_{*}= 0.$
Using the fact that $u_n\in \mathcal{N}_{\lambda},$ we have that $\|u_n \|_{c_2,q}^q=\lambda K_q(u_n)-\|u_n \|_{c_1,p}^p.$ Since $\lim_{n\to\infty}K_q(u_n)=K_q(0)=0$, then by  \eqref{K_p} we get \[\lim_{n\to\infty} \|u_n \|_{c_2,q}^q=\lambda\lim_{n\to\infty} K_q(u_n)-\lim_{n\to\infty} \|u_n \|_{c_1,p}^p=-\frac{qp}{q-p}m_{\lambda}<0.\]
This leads to a contradiction. Thus, $u_{*}\neq 0.$\\
Now, let us prove that  $\|u_*\|_{c_1,p}^p+\|u_* \|_{c_2,q}^q=\lambda K_q(u_{*})$. Assume by contradiction that the equality does not hold. Then by \eqref{deskpkq}, 
\begin{equation}\label{hpc}
\|u_*\|_{c_1,p}^p+\|u_* \|_{c_2,q}^q<\lambda K_q(u_{*}).
\end{equation}
We claim that there is a constant $\tau \in(0,1)$ such that $\tau u_{*} \in \mathcal{N}_{\lambda}$. To see this,  define a continuous function $g:(0, \infty) \rightarrow \mathbb{R}$ by
$
 g(t):=t^{p-q}\|u_*\|_{c_1,p}^p+\|u_* \|_{c_2,q}^q-\lambda K_q(u_{*}).
$

As $\|u_*\|_{c_1,p}^p>0$ and $p<q$, we have $\lim_{t \rightarrow 0^{+}}g(t)=\infty$.  By \eqref{hpc} we notice that $g(1)<0$. Therefore, there exists $\tau \in(0,1)$ such that $g(\tau)=0$ which implies $\tau u_{*} \in \mathcal{N}_{\lambda}$.
Since $u_{n} \rightharpoonup u_{*} \in W^{1,p}(\Omega)$, we have 
\begin{equation}\label{semi}
    \frac{q-p}{pq}\|u_*\|_{c_1,p}^p \leqslant \liminf_{n\to \infty} \frac{q-p}{pq}\|u_n\|_{c_1,p}^p .
\end{equation}
By \eqref{J_rest_eta}, \eqref{Jlambda}, \eqref{semi} and the fact that $\tau\in (0,1)$ we obtain
$$
m_{\lambda} \leqslant J_{\lambda}\left(\tau u_{*}\right)
\leqslant \tau^{p} \lim _{n \rightarrow \infty} J_{\lambda}\left(u_{n}\right)=\tau^{p} m_{\lambda}<m_{\lambda}.
$$
This is a contradiction. So, 
\begin{equation*}\label{deskp}
\|u_*\|_{c_1,p}^p+\|u_* \|_{c_2,q}^q=\lambda K_q(u_{*}).
\end{equation*}
This implies that  $u_{*}\in \mathcal{N}_{\lambda}$.  By \eqref{Jlambda} and \eqref{semi}, we have

$$m_{\lambda} \leqslant J_{\lambda}\left(u_{*}\right)= \frac{q-p}{q p} \|u_*\|_{c_1,p}^p \leqslant \liminf_{n\to \infty} \frac{q-p}{pq}\|u_n\|_{c_1,p}^p = m_{\lambda}.
$$
Thus $J_{\lambda}\left(u_{*}\right)=m_{\lambda}$. This completes the proof.
\end{itemize}
\end{proof}

\begin{lem}\label{lagrange_lambda} 
    If $p<q$ and $\lambda >\Lambda_q$ then $\lambda$ is an eigenvalue of Problem \eqref{eig-prob_1}.
\end{lem}
\begin{proof}
   Fix $\lambda>\Lambda_q$. We assert that the minimizer $u_*\in \mathcal{N}_\lambda$ obtained in Lemma \ref{2-exist-eigen} is an eigenfunction of Problem \eqref{eig-prob_1} corresponding to the eigenvalue $\lambda$.
    Clearly $u_*$ is a solution to the following constrained problem
    \begin{equation*}
        \min_{v\in W\setminus\{0\}}J_\lambda(v), \quad g_q(v):=\|v\|_{c_1,p}^p+\|v\|_{c_2,q}^q-\lambda K_q(v)=0.
    \end{equation*}
  Observe that given $\xi \in \mathbb{R}$, there exists $\chi \in \mathbb{R}$ such that $\xi=\langle g'_q(u_*), \chi u_*\rangle$ if and only if 
    \begin{equation}\label{lagrange-3}
        \xi = \chi \left(p\|u_*\|_{c_1, p}^p+q\|u_*\|_{c_2, q}^q-\lambda qK_q(u_*)\right).
    \end{equation}
    Since \begin{equation}\label{u_*-neq}
        p\|u_*\|_{c_1, p}^p+q\|u_*\|_{c_2, q}^q-\lambda qK_q(u_*)=(p-q)\|u_*\|_{c_1, p}^p\neq 0 ,
    \end{equation}  then equation \eqref{lagrange-3} is equivalent to 
    \begin{equation*}
     \chi =\frac{\xi}{(p-q)\|u_*\|_{c_1, p}^p}.   
    \end{equation*}
    From this chain of equivalences we infer that the linear functional $g'_q(u_*):W \to \mathbb{R}$ is surjective. 
    From [\cite{GP}, Theorem 5.5.26, p. 701] there exist $\alpha, \mu \in \mathbb{R}$, with $(\alpha, \mu)\neq (0,0)$ such that \begin{equation}\label{lagrange-2}
        \alpha\langle J'_\lambda (u_*), v\rangle+\mu \langle g'_q(u_*), v \rangle=0, \qquad \text{for all $v\in W$}.
    \end{equation} 
    In particular, with $v=u_*$ and taking into account that $u_*\in\mathcal{N}_\lambda$ we get
    \begin{equation}\label{lagrange-1}
        \mu \left(p\|u_*\|_{c_1, p}^p+q\|u_*\|_{c_2, q}^q-\lambda qK_q(u_*)\right)=0.
    \end{equation}
From \eqref{u_*-neq} and \eqref{lagrange-1}, we see that $\mu=0$. Going back to \eqref{lagrange-2} we have
    \begin{equation*}
      \alpha\langle J'_\lambda (u_*), v\rangle=0,  \qquad \text{for all $v\in W$}.
    \end{equation*}
   In view of the fact that $(\alpha, \mu)\neq (0,0)$ and $\mu=0$ leads us to 
    \begin{equation*}
        \langle J'_\lambda (u_*), v\rangle=0,  \qquad \text{for all $v\in W$},
    \end{equation*}
    i.e. $u_*\neq 0$ is an eigenfunction of Problem \eqref{eig-prob_1} corresponding to $\lambda$. Thus $\lambda>\Lambda_q$ is an eigenvalue of Problem \eqref{eig-prob_1}.
\end{proof}

\vskip0.4cm 
\noindent
{\bf Proof of Theorem  \ref{Neumman_theo}.}\mbox{}
\vskip0.2cm 
\noindent
    Let $\lambda \in \mathbb{R}$. If $\lambda \leqslant \Lambda_q$, then, by Lemma \ref{0<lambda_q} (iii), $\lambda$ is not an eigenvalue of Problem \eqref{eig-prob_1}. On the other hand, if $\lambda >\Lambda_q$, then, by Lemma \ref{1-exist-eigen}, $\lambda$ is an eigenvalue of Problem \eqref{eig-prob_1} in the case $q<p$, and the same result follows from Lemma \ref{lagrange_lambda} when $q>p$. \qed

\vskip0.4cm 
\section{Nonlinear Perturbation} 
In this section we prove Theorems \ref{main_theo} and \ref{main_theo2}  through variational arguments (see \cite[Theorem 2.7]{PR}). In order to apply  \cite[Theorem 2.7]{PR}, we need to the Palais-Smale (PS) condition. In what follows,  we first state and prove a precompactness result similar to \cite[Proposition B.35]{PR} that will allow us to verify (PS) in our setting.\\

Consider the functional $J:W \to \mathbb{R}$ defined by \begin{equation}\label{def_J} 
    J(u):= \frac{1}{p}\|u\|_{c_1,p}^p+\frac{1}{q}\|u\|_{c_2,q}^q-\int_\Omega F(x,u)  - \int_{\partial\Omega} G(x,u).
\end{equation}
Due to our hypotheses on $m, \rho, f$ and $g$, and Propositions \ref{prop-nabla} and \ref{prop-cu} we have that $J\in C^1$  and 
\begin{equation}\label{form_J'}    \langle J' (u), \varphi  \rangle= B(u, \varphi)-\int_\Omega f(x,u)\varphi - \int_{\partial\Omega} g(x,u)\varphi  , \quad \text{ for all $\varphi \in W$},
\end{equation}
where 
\begin{equation}\label{B(u,phi)}
    B(u, \varphi):=\int_\Omega |\nabla u|^{p-2}\nabla u\cdot \nabla \varphi +\int_\Omega c_1(x)| u|^{p-2} u \varphi +\int_\Omega |\nabla u|^{q-2}\nabla u\cdot \nabla \varphi+\int_\Omega c_2(x) | u|^{q-2} u \varphi  .
\end{equation}

Some  steps in the proof of the following proposition are similar to those contained in \cite[Lemma 2.3]{CdFL}.
\begin{proposition}\label{bound-PS}
Assume that  conditions (C1) and (C3)-(C5) hold and let $J$ be defined as \eqref{def_J}. If $(u_n)$ is a bounded sequence in $W$ such  that $J'(u_n) \to 0$ as $n\to \infty$. Then  $(u_n)$ has a convergent subsequence.
\end{proposition}
\begin{proof}
    Observe that $W$ is reflexive  and is compactly embedded in $L^{m_1}(\Omega)$ and  in $ L^{m_2}(\partial\Omega)$ with any $1\le m_1<r^*$ and $1\le m_2<r_*$, where $r^*$ and $r_*$ are defined in $(C4)$ and $(C5)$. Since $(u_n)$ is bounded in $W$, then there exists a subsequence (which we shall denote again $(u_n)$) such that $u_n \rightharpoonup u_*$ in $W$, $u_n \to u_*$ in $L^p(\Omega)$, $u_n \to u_*$ in $L^q(\partial\Omega)$  and $u_n \to u_*$ in $L^q(\Omega)$.
 Using the fact that $J'(u_n)\to 0$ we have that 
 \begin{equation}\label{J'(u_n)}
     \langle J'(u_n), u_n-u_* \rangle \to 0,\qquad \text{as $n\to \infty$}.
 \end{equation}
Notice that from (C4)  with $\alpha_1:=(s_1+1)'$ (the Hölder conjugate of $s_1+1$ and therefore $\alpha_1 s_1=s_1+1$) and taking into account that $(u_n)$ is bounded in $L^{s_1+1}(\Omega)$ we have
 \begin{eqnarray}\label{f_alpha_1}
     \int_\Omega |f(x,u_n)|^{\alpha_1}  
     &\leqslant & \hat{c}|\Omega|+\hat{c}\hat{a}\|u_n\|_{s_1+1}^{s_1+1}\leqslant C. 
 \end{eqnarray}
 for some constant $C$ independent of $n$.
 Therefore, applying Hölder inequality with $\alpha_1$ and using \eqref{f_alpha_1} we have
 \begin{eqnarray}\label{int_f_1}
\left|     \int_{\Omega} f(x, u_n) (u_n-u_*)  \right| 
&\leqslant & C\|u_n-u_*\|_{s_1+1}^{s_1+1} .
 \end{eqnarray}
Similarly, since $(u_n)$ is bounded in $L^{s_2+1}(\partial \Omega)$, from (C5) with $\alpha_2:=(s_2+1)'$ we get 
 \begin{eqnarray}\label{int_f_2}
\left|     \int_{\partial\Omega} g(x, u_n) (u_n-u_*)  \right| 
&\leqslant & C\|u_n-u_*\|_{\partial\Omega, s_2+1}^{s_2+1}
,
 \end{eqnarray}
for some constant $C$ independent of $n$.
From \eqref{int_f_1}
 and \eqref{int_f_2}   and taking into account that $u_n \to u_*$ in $L^{s_1+1}(\Omega)$ and $u_n \to u_*$ in $L^{s_2+1}(\partial\Omega)$, we have  
\begin{equation}\label{f_g}
\int_\Omega f(x, u_n) (u_n-u_*)  \to 0 \qquad \text{and} \qquad      \int_{\partial\Omega} g(x, u_n) (u_n-u_*)  \to 0 ,\qquad \text{as $n\to \infty$}.
 \end{equation}    Using \eqref{form_J'} with $\varphi=u_n-u_*$, \eqref{J'(u_n)} and  \eqref{f_g} , we see that
     \begin{equation}\label{B_goes_0_1}
         B(u_n, u_n-u_*) \to 0,\qquad \text{as $n\to \infty$}.
     \end{equation}
Furthermore,     by the Hölder inequality, we have
     \begin{equation*}
  \left|       \int_\Omega c_1(x) |u_n|^{p-2}u_n(u_n-u_*)  \right| 
  \leqslant \|c_1\|_\infty\|u_n\|_p^{p-1}\|u_n-u_*\|_p
.     \end{equation*}
Since  $u_n \to u_*$ in $L^p(\Omega)$, we see that $\displaystyle \int_\Omega c_1(x)|u_n|^{p-2}u_n(u_n-u_*)  \to 0$. Similarly, we get that
$\displaystyle \int_\Omega c_2(x)|u_n|^{q-2}u_n(u_n-u_*)  \to 0$. Hence, from \eqref{B(u,phi)} and \eqref{B_goes_0_1}, we obtain
\begin{equation}\label{grad(u_n-u)_1}
    \int_\Omega |\nabla u_n|^{p-2} \nabla u_n \cdot (\nabla u_n-\nabla u_*)  +     \int_\Omega |\nabla u_n|^{q-2} \nabla u_n \cdot (\nabla u_n-\nabla u_*)  \to 0, \qquad \text{as }n\to \infty.
\end{equation}

On the other hand, since $u_n \rightharpoonup u_* $ in $W$  then, for any $\theta \in \{p,\, q\} $ 
\begin{equation}\label{grad(u_n-u)_2}
       \int_\Omega |\nabla u_*|^{\theta-2} \nabla u_* \cdot (\nabla u_n-\nabla u_*)  \to 0 \qquad \text{as }n\to \infty .
\end{equation}
Additionally, using Hölders inequalities, we get
\begin{align*}\label{|u_n-u|1}
    \int_\Omega (|\nabla u_n|^{\theta-2} \nabla u_n-|\nabla u_*|^{\theta-2} \nabla u_* )\cdot (\nabla u_n-\nabla u_*)  
      &\geqslant (\|\nabla u_n\|_\theta^{\theta-1}-\|\nabla u_*\|_\theta^{\theta-1})(\|\nabla u_n\|_\theta-\|\nabla u_*\|_\theta).
\end{align*}
Then
\begin{eqnarray}\label{chain_1}
    &&\int_\Omega (|\nabla u_n|^{p-2} \nabla u_n-|\nabla u_*|^{p-2} \nabla u_* )\cdot (\nabla u_n-\nabla u_*)  \nonumber +  \\
    && \int_\Omega (|\nabla u_n|^{q-2} \nabla u_n -|\nabla u_*|^{q-2} \nabla u_*) \cdot (\nabla u_n-\nabla u_*)  \nonumber\\
    && \geqslant (\|\nabla u_n\|_p^{p-1}-\|\nabla u_*\|_p^{p-1})(\|\nabla u_n\|_p-\|\nabla u_*\|_p)+\nonumber (\|\nabla u_n\|_q^{q-1}-\|\nabla u_*\|_q^{q-1})(\|\nabla u_n\|_q-\|\nabla u_*\|_q) \\
    &&
    \geqslant  0.
\end{eqnarray}
The left hand side of \eqref{chain_1} goes to $0$ as $n\to \infty$ due to \eqref{grad(u_n-u)_1} and \eqref{grad(u_n-u)_2} with $\theta =p,q$. Therefore 
\begin{equation}\label{Lim_norm_2}
    \lim_{n\to \infty }(\|\nabla u_n\|_p^{p-1}-\|\nabla u_*\|_p^{p-1})(\|\nabla u_n\|_p-\|\nabla u_*\|_p)=0
\end{equation}
and 
\begin{equation}\label{lim_norm_3}
    \lim_{n\to \infty }(\|\nabla u_n\|_q^{q-1}-\|\nabla u_*\|_q^{q-1})(\|\nabla u_n\|_q-\|\nabla u_*\|_q)=0.
\end{equation}

 We have from \eqref{Lim_norm_2} and \eqref{lim_norm_3} that
\begin{equation}\label{norm_grad_1}
    \lim_{n\to \infty}\|\nabla u_n\|_p = \|\nabla u_*\|_p \quad \text{and} \quad \lim_{n\to \infty}\|\nabla u_n\|_q = \|\nabla u_*\|_q. 
\end{equation}

Combining  Proposition \ref{prop-cu} and the fact that $u_n \to u_*$ in $L^p(\Omega)$, it follows that 

\begin{equation}\label{norm_c1p_1}
    \lim_{n\to \infty} \int_\Omega c_1 (x)|u_n|^p  = \int_\Omega c_1 (x)|u_*|^p 
\end{equation}
and
\begin{equation}\label{norm_c2p_2}
    \lim_{n\to \infty} \int_\Omega c_2 (x)|u_n|^q  = \int_\Omega c_2 (x)|u_*|^q .
\end{equation}
From \eqref{norm_grad_1}, \eqref{norm_c1p_1} and \eqref{norm_c2p_2} we get
\begin{equation}\label{conv_norm_1}
    \lim_{n\to \infty} \|u_n\|_{c_1,p}=\|u_*\|_{c_1,p} \quad \text{and} \quad     \lim_{n\to \infty} \|u_n\|_{c_2,q}=\|u_*\|_{c_2,q}.
\end{equation}
Taking into account that $W$ is uniformly convex (see  \cite[Proposition 2.4.7]{GP2}),  \eqref{conv_norm_1} and the fact that $u_n \rightharpoonup u_*$ in $W$, we conclude from \cite[Proposition 3.32]{brezis} that $u_n \to u_*$ (strongly) in $W$. 
\end{proof}

Now, we can prove our main results. 

\vskip0.4cm 
\noindent
{\bf Proof of Theorem  \ref{main_theo}}\mbox{}
\vskip0.3cm
\noindent
We will prove the theorem in three steps.\\
\vskip0.3cm
\noindent
\textbf{Step 1}:  $J$ is coercive, that is, $J(u)\to \infty$ as $\|u\|_{W}\to \infty$. Define $\kappa:=\max\{\lambda, \mu\}$.   By hypothesis $\kappa<\Lambda_q$. From the continuous embedding of $W$ into $L^q(\Omega)$ and the continuity of the trace operator from $W$ into $L^q(\partial \Omega)$ (see \eqref{trace_ineq}), there exists $C_q>0$ such that 
\begin{equation}\label{embedd1}
\|u\|_q^q+\|u\|_{\partial\Omega, q}^q \leqslant C_q \|u\|_{c_2, q}^q, \quad \text{for all } u\in W,\end{equation}
where $C_q$ is a constant independent of $u$. Since  $\kappa<\Lambda_q$,
then there exists $\varepsilon >0$ such that 
\begin{equation}\label{eps_kap_lam}
    1-\varepsilon C_q -\frac{\kappa}{\Lambda_q}>0.
\end{equation}
From \eqref{hypoth_FG}  and conditions (C3)--(C5), there exist 
      $M_{\varepsilon}>0$ such that for all  $x\in \overline{\Omega}$ and all $u\in \mathbb{R}$ 
    \begin{eqnarray}\label{qF_1}
        F(x,u) &\leqslant &\frac{1}{q}(\lambda m(x) + \varepsilon)|u|^q +M_\varepsilon \qquad \text{and} \\
        G(x,u) &\leqslant & \frac{1}{q}(\mu \rho (x) + \varepsilon)|u|^q +M_{\varepsilon} .\label{qG_2}
    \end{eqnarray}
Therefore,
    from \eqref{def_J}, \eqref{qF_1} and \eqref{qG_2} we have   \begin{equation}\label{J_coerciv_1a}
      \begin{aligned}
        J(u) &\geqslant  
        \frac{1}{p}\|u\|_{c_1,p}^p+\frac{1}{q}\|u\|_{c_2,q}^q- \frac{\kappa}q \left(\int_\Omega  m(x) |u|^q   +\int_{\partial \Omega} \rho(x) |u|^q\right) \\
        &  -\frac{\varepsilon}{q}(\|u\|_q^q+\|u\|_{\partial\Omega ,q}^q)- \hat{M}_\varepsilon, 
              \end{aligned}
    \end{equation}
where $\hat{M}_\varepsilon:=\frac{M_\varepsilon}{q}(|\Omega|+|\partial\Omega|)$.

Using \eqref{embedd1}, \eqref{Lambda_q} and \eqref{def-K_q}, we get
\begin{equation}\label{Ju_upperbound}
    \begin{aligned}
        J(u)         &  \geqslant \frac{1}{p}(1-\varepsilon C_q)\|u\|_{c_1,p}^p+\frac{1}{q}(1-\varepsilon C_q)\|u\|_{c_2,q}^q- \frac{\kappa}q K_q(u)  - \hat{M}_\varepsilon\\     
        & \geqslant \left(1-\varepsilon C_q-\frac{\kappa}{ \Lambda_q}\right)\left(\frac{1}{p}\|u\|_{c_1,p}^p+\frac{1}{q}\|u\|_{c_2,q}^q\right) - \hat{M}_\varepsilon.
    \end{aligned}
\end{equation}
By \eqref{eps_kap_lam}, we have that  $1-\varepsilon C_q-\frac{\kappa}{ \Lambda_q}>0$. It follows from \eqref{Ju_upperbound} that either $J(u)\ge \left(1-\varepsilon C_q-\frac{\kappa}{ \Lambda_q}\right)\frac{1}{p}\|u\|_{c_1,p}^p - \hat{M}_\varepsilon$ if $u\in W=W^{1,p}(\Omega)$ or $J(u)\ge \left(1-\varepsilon C_q-\frac{\kappa}{ \Lambda_q}\right)\frac{1}{q}\|u\|_{c_2,q}^q  - \hat{M}_\varepsilon$ if $u\in W=W^{1,q}(\Omega)$. In either case, $J(u)\to \infty$ as $\|u\|\to \infty$. Hence,   $J$ is coercive.
\vskip0.2cm
\noindent
\textbf{Step 2:} We will prove that $J$ satisfies the Palais-Smale condition. Let $(u_n)$ be  a sequence  in $W$   such that $(J(u_n))$  is bounded and $ J' (u_n) \to 0$ as $n\to \infty$. We need to show that $(u_n)$ has a convergent subsequence.  In view of Proposition \ref{bound-PS}, it suffices to prove that $(u_n)$ is bounded in $W$. Let $M_1>0$ be such that $|J(u_n)|\leqslant M_1$  for all $n\in \mathbb{N}$. Therefore, using the definition of $J$, we have that 
\begin{equation}\label{J_bounded1}
     \frac{1}{p}\|u_n\|_{c_1, p}^p+\frac{1}{q}\|u_n\|_{c_2, q}^q\leqslant  \int_\Omega F(x,u_n) +\int_{\partial\Omega} G(x,u_n)  + M_1  .
 \end{equation}
We claim that $(u_n)$ is bounded in $W$.
From  \eqref{embedd1}, \eqref{qF_1}, \eqref{qG_2}
and \eqref{J_bounded1}, the definition of $\tilde\Lambda_q$ given in \eqref{Lambda_q2} and Lemma \ref{0<lambda_q} (ii) and taking into account the fact that $\kappa=\max\{\lambda, \mu\} <\Lambda_q$, we get
\begin{equation}\label{J_bounded2}
\begin{aligned}
        \frac{1}{p}\|u_n\|_{c_1, p}^p+\frac{1}{q}\|u_n\|_{c_2, q}^q & \leqslant 
        \frac{\lambda}{q} \int_\Omega m(x)|u_n|^q +\frac{\mu}{q} \int_{\partial\Omega}  \rho(x)|u_n|^q + \frac{\varepsilon}{q}(\|u_n\|_q^q+\|u_n\|_{\partial\Omega,q}^q) +\tilde{M}_\varepsilon 
        \\
        &\leqslant \frac{\kappa}{\Lambda_q} \left(  \frac{1}{p}\|u_n\|_{c_1, p}^p+\frac{1}{q}\|u_n\|_{c_2, q}^q \right) +\varepsilon C_q\left(  \frac{1}{p}\|u_n\|_{c_1, p}^p+\frac{1}{q}\|u_n\|_{c_2, q}^q \right)+\tilde{M}_\varepsilon.
 \end{aligned}
\end{equation}
where $\tilde{M}_\varepsilon:=M_\varepsilon |\Omega| + M_\varepsilon |\partial\Omega|+M_1 $. Consequently, from \eqref{J_bounded2}
\begin{equation}\label{un-bound2}
    \begin{aligned}
        \left(1-\frac{\kappa}{\Lambda_q}-\varepsilon C_q\right)\left(  \frac{1}{p}\|u_n\|_{c_1, p}^p+\frac{1}{q}\|u_n\|_{c_2, q}^q \right)\leqslant \tilde{M}_\varepsilon.
    \end{aligned}
\end{equation}
From \eqref{eps_kap_lam} and \eqref{un-bound2} we see that $(u_n)$ is bounded in $W$. By Proposition \ref{bound-PS} it follows that $J$ satisfies the Palais-Smale condition.
\vskip0.2cm
\noindent
\textbf{Step 3:} We will prove that Problem \eqref{nonlinpde} has at least one weak solution.
Indeed, by Propositions \ref{prop-nabla} and \ref{prop-cu}, and hypotheses (C3)-(C5) we have that $J\in C^1$. Furthermore, $J$ is coercive, by Step 1, and by Step 2 $J$ satisfies the Palais-Smale condition. Therefore, by virtue of \cite[Theorem 2.7]{PR} $J$ has a critical point $u_0$. Moreover, $\displaystyle J(u_0)=\inf_W J$. Thus, $u\in W$  is a weak solution of \eqref{nonlinpde}. This completes the proof. \qed

\vskip0.4cm 
\noindent
{\bf Proof of Theorem  \ref{main_theo2}}\mbox{}

\noindent

We will prove it in three steps.

\noindent
\textbf{Step 1}:  $J$ is coercive. 
     By hypotheses \eqref{hypoth_FG_2} and \eqref{hypoth_FG_3}, there exists $\varepsilon >0$ such that 
    \begin{equation}\label{restric_2}
    \lambda +\varepsilon <\Pi_q, \qquad \mu +\varepsilon <\Gamma_q   
    \end{equation}
     and  \begin{equation}\label{restric_1}
        1-\frac{\lambda+\varepsilon}{{\Pi}_q}-\frac{\mu+\varepsilon}{{\Gamma}_q}>0.
    \end{equation}

From \eqref{hypoth_FG_2}  and conditions (C3)--(C5),  there exists $M_{\varepsilon}>0 $ such that for all $x\in \Omega$ and for all $u\in \mathbb{R}$
    \begin{eqnarray}\label{qF}
        F(x,u) &\leqslant &\frac{1}{q}(\lambda + \varepsilon)|u|^q +{M}_\varepsilon \qquad \hbox{and}\\
        G(x,u) &\leqslant & \frac{1}{q}(\mu + \varepsilon )|u|^q+ M_{\varepsilon}.\label{qG}
    \end{eqnarray}
Using \eqref{def_J}, \eqref{qF}  and \eqref{qG} we have that
\begin{equation}\label{J_coerciv_1}
        J(u) \geqslant  \frac{1}{p}\|u\|_{c_1,p}^p+\frac{1}{q}\|u\|_{c_2,q}^q - \frac{\lambda+\varepsilon}{q} \|u\|_q^q- \frac{\mu+\varepsilon}{q}\|u\|_{\partial \Omega,q}^q - M_\varepsilon(|\Omega|+|\partial\Omega|).
    \end{equation}

Observe that
 the definition \ref{Lambda_q} and Lemma \ref{0<lambda_q} (ii) implies that
  for all $u\in W$
\begin{equation}\label{J_coerciv_2}
   \frac{\lambda+\varepsilon}{q} \|u\|_q^q\leqslant \frac{\lambda+\varepsilon}{{\Pi}_q}\left(\frac{1}{p}\|u\|_{c_1,p}^p+\frac{1}{q}\|u\|_{c_2,q}^q\right)
 \end{equation}
whenever $\lambda+\varepsilon\geqslant 0$, and
\begin{equation}\label{J_coerciv_3}
     \frac{\mu+\varepsilon}{q} \|u\|_{\partial \Omega,q}^q \leqslant \frac{\mu+\varepsilon}{{\Gamma}_q}\left(\frac{1}{p}\|u\|_{c_1,p}^p+\frac{1}{q}\|u\|_{c_2,q}^q\right) ,
\end{equation}
 provided that $\mu+\varepsilon \geqslant 0$.
    To prove the coercivity of $J$ we shall analyze 4 cases. 
    \begin{itemize}
        \item[Case 1:] $\lambda+\varepsilon, \mu+\varepsilon \geqslant 0$.\\
From \eqref{J_coerciv_1}, \eqref{J_coerciv_2} and \eqref{J_coerciv_3}, we obtain
\begin{align}\label{J_coerciv_4}
    J(u) &\geqslant 
    \frac{1}{p}\left( 1- \frac{\lambda+\varepsilon}{ {\Pi}_q}-\frac{\mu+\varepsilon}{ {\Gamma}_q} \right)\|u\|_{c_1,p}^p+\frac{1}{q}\left( 1- \frac{\lambda+\varepsilon}{ {\Pi}_q}-\frac{\mu+\varepsilon}{ {\Gamma}_q} \right)\|u\|_{c_2,q}^q-C_\varepsilon,
\end{align}
where $C_\varepsilon:=M_\varepsilon (|\Omega|+|\partial\Omega|)$. If $p > q$, then $r=p$ and $\|u\|_{c_1,p}=\|u\|_{W}$. On the other hand if $p<q$ then $r=q$ and $\|u\|_{c_2,q}=\|u\|_{W}$. 
Taking into account \eqref{restric_1}, we obtain from \eqref{J_coerciv_4} that \begin{equation*}\label{J_coerciv_5}
     J(u) \geqslant \frac{1}{r}\left( 1- \frac{\lambda+\varepsilon}{ {\Pi}_q}-\frac{\mu+\varepsilon}{ {\Gamma}_q} \right)\|u\|_{W}^r-C_\varepsilon.
 \end{equation*}
 Hence $J(u) \to \infty$ as $\|u\|_{W}\to \infty$. Thus $J$ is coercive.  
\item[Case 2:] $\lambda+\varepsilon \geqslant 0$ and $\mu+\varepsilon<0$.\\ 
From \eqref{J_coerciv_1} we have
 \begin{equation*}
     J(u) \geqslant \frac{1}{p}\|u\|_{c_1,p}^p+\frac{1}{q}\|u\|_{c_2,q}^q - \frac{\lambda+\varepsilon}{q} \|u\|_q^q -C_\varepsilon.
 \end{equation*}
Applying \eqref{J_coerciv_2} we get  \begin{equation}\label{J_coerciv_6}
   J(u)\geqslant  \left( 1- \frac{\lambda+\varepsilon}{ {\Pi}_q}\right)\frac{1}p\|u\|_{c_1,p}^p+\left( 1- \frac{\lambda+\varepsilon}{ {\Pi}_q}\right)\frac{1}{q}\|u\|_{c_2,q}^q-C_\varepsilon .
\end{equation}
 By \eqref{restric_2}, we have that $ 1- \frac{\lambda+\varepsilon}{ {\Pi}_q}>0$. Then arguing as in case 1 we have that $J$ is coercive. 
 \item[Case 3:]
 $\mu +\varepsilon \geqslant 0$ and $\lambda +\varepsilon <0$. Arguing as in the case 2, we obtain \eqref{J_coerciv_6} which implies the coercivity of $J$.\\ 
\item[Case 4:]  $\lambda+\varepsilon, \mu+\varepsilon < 0$.\\ 
From \eqref{J_coerciv_1} we obtain
 \begin{equation*}
     J(u) \geqslant \frac{1}{p}\|u\|_{c_1,p}^p+\frac{1}{q}\|u\|_{c_2,q}^q  -C_\varepsilon.
 \end{equation*}
 Then, arguing as in case 1 we conclude that $J$  is coercive. 
    \end{itemize}
Each of these cases led to   $J$ being  coercive.
\vskip0.2cm
\noindent
\textbf{Step 2}:  $J$ satisfies the Palais-Smale condition.
Indeed, let $(u_n)$ be a sequence in $W$   such that $(J(u_n))$  is bounded and $ J' (u_n) \to 0$, as $n\to \infty$.  By Proposition \ref{bound-PS}, it suffices to prove that $(u_n)$ is bounded in $W$. Since $(J(u_n))$  is bounded then there exists a constant $M_1>0$ such that $|J(u_n)|\leqslant M_1$  for all $n\in \mathbb{N}$. 
\\
Using \eqref{def_J} we get
\begin{equation}\label{J_bounded}
     \frac{1}{p}\|u_n\|_{c_1, p}^p+\frac{1}{q}\|u_n\|_{c_2, q}^q\leqslant  \int_\Omega F(x,u_n) +\int_{\partial\Omega} G(x,u_n)  + M_1  .
 \end{equation}
 In order to prove that $(u_n)$ is bounded, we shall consider 4 cases. As in the proof of the coercivity of $J$ we take $\varepsilon >0$  that  satisfies \eqref{restric_2} and \eqref{restric_1}. 
\begin{itemize}
    \item[Case 1:] 
$\lambda+\varepsilon,\mu+\varepsilon\geqslant 0$.\\
Using \eqref{J_bounded} combined with \eqref{qF}, \eqref{qG},  \eqref{J_coerciv_2} and \eqref{J_coerciv_3}, we get
 \begin{eqnarray*}
     \frac{1}{p}\|u_n\|_{c_1, p}^p+\frac{1}{q}\|u_n\|_{c_2, q}^q 
     &\leqslant & \frac{\lambda+\varepsilon}{q} \|u_n\|_q^q+\frac{\mu+\varepsilon}{q} \|u_n\|_{\partial \Omega,q}^q+\frac{M\varepsilon}{q}|\Omega|+\frac{M\varepsilon}{q}|\partial\Omega|+M_1\\
     &\leqslant & \left( \frac{\lambda+\varepsilon}{ {\Pi}_q}+\frac{\mu+\varepsilon}{ {\Gamma}_q} \right)\left( \frac{1}{p} \|u_n\|_{c_1, p}^p+\frac{1}{q} \|u_n\|_{c_2, q}^q\right)+M_2 ,
 \end{eqnarray*}
 where $M_2:=\frac{M\varepsilon}{q}|\Omega|+\frac{M\varepsilon}{q}|\partial\Omega|+M_1$.
 Therefore
 \begin{align}\label{u_n-bound1}
 \frac{1}r \left(1-\frac{\lambda+\varepsilon}{ {\Pi}_q}-\frac{\mu+\varepsilon}{ {\Gamma}_q}\right) \|u_n\|_{W}^r  \leqslant  \left( 1-\frac{\lambda+\varepsilon}{ {\Pi}_q}-\frac{\mu+\varepsilon}{ {\Gamma}_q} \right) \left( \frac{1}{p} \|u_n\|_{c_1, p}^p +  \frac{1}{q} \|u_n\|_{c_2, q}^q \right) \leqslant  M_2.
 \end{align}
Observe that $1-\frac{\lambda+\varepsilon}{ {\Pi}_q}-\frac{\mu+\varepsilon}{ {\Gamma}_q}>0$ as \eqref{restric_1} holds. Therefore, 
\eqref{u_n-bound1} implies that $(u_n)$ is bounded in $W$. \\

\item[Case 2:]  $\lambda+\varepsilon \geqslant 0$ and $\mu+\varepsilon <0$.\\
From   \eqref{qG} and \eqref{J_bounded} we have that
\begin{eqnarray*}
       \frac{1}{p}\|u_n\|_{c_1, p}^p+\frac{1}{q}\|u_n\|_{c_2, q}^q &\leqslant &  
       \frac{\lambda +\varepsilon}{q}\|u_n\|_q^q  + M_\varepsilon|\Omega| +|\partial\Omega| M_\varepsilon  + M_1.
\end{eqnarray*}
Using the fact $\lambda+\varepsilon \geqslant0$ and  \eqref{J_coerciv_2} holds, we get that 
\begin{equation*}
     \frac{1}{p}\|u_n\|_{c_1, p}^p+\frac{1}{q}\|u_n\|_{c_2, q}^q \leqslant  \frac{\lambda+\varepsilon}{ {\Pi}_q}\left(\frac{1}{p}\|u_n\|_{c_1,p}^p +\frac{1}{q}\|u_n\|_{c_2,q}^q\right)+M_3,
\end{equation*}
where $M_3:=M_\varepsilon |\Omega| +|\partial\Omega| M_\varepsilon  + M_1.$
We   obtain
\begin{equation*}
    \frac{1}r \left(1-\frac{\lambda+\varepsilon}{ {\Pi}_q}\right) \|u_n\|_{W}^r\leqslant   \left(1-\frac{\lambda+\varepsilon}{ {\Pi}_q}\right)   \left(\frac{1}{p}\|u_n\|_{c_1,p}^p +\frac{1}{q}\|u_n\|_{c_2,q}^q\right)\leqslant M_3, 
\end{equation*}
Observe that $1-\frac{\lambda+\varepsilon}{ {\Pi}_q}>0$  by 
 \eqref{restric_1}. Thus, $(u_n)$ is bounded in $W$.
\\
\item[Case 3:]  
$\lambda+\varepsilon < 0$ and $\mu+ \varepsilon\geqslant 0$. Arguing as in the case 2, we obtain,
 with $M_4:=\frac{M\varepsilon}{q}|\Omega| +|\partial\Omega| \frac{M\varepsilon}{q}  + M_1$,
\begin{equation*}       \frac{1}r\left(1-\frac{\mu+\varepsilon}{ {\Gamma}_q} \right) \|u_n\|_{W}^r\leqslant\left(1-\frac{\mu+\varepsilon}{ {\Gamma}_q}\right)   \left(\frac{1}{p}\|u_n\|_{c_1,p}^p +\frac{1}{p}\|u_n\|_{c_2,q}^q\right)\leqslant M_4.
\end{equation*}
Note that $1-\frac{\mu+\varepsilon}{ {\Gamma}_q}>0$ by 
 \eqref{restric_1}. Thus, $(u_n)$ is bounded in $W$.

\item[Case 4:]
 $\lambda+\varepsilon<0$ and $\mu+\varepsilon<0.$\\  From \eqref{qF}, \eqref{qG} and \eqref{J_bounded}, we have 
\begin{equation*}
 \frac{1}r \|u_n\|_{W}^r\leqslant   \frac{1}{p}\|u_n\|_{c_1, p}^p+\frac{1}{q}\|u_n\|_{c_2, q}^q \leqslant  M_4. 
\end{equation*}
Therefore $(u_n)$ is bounded in $W$. 
\end{itemize}
Since $(u_n)$ is bounded in $W$, it follows from Proposition \ref{bound-PS}  that $J$ satisfies the Palais-Smale condition. 

\vskip0.2cm
\noindent
\textbf{Step 3}: We show that Problem \eqref{nonlinpde} has at least one weak solution.
Indeed, by Propositions \ref{prop-nabla} and \ref{prop-cu}, and hypotheses $(C3)-(C5)$ we have that $J\in C^1$. Furthermore, $J$ is coercive, by Step 1, and by Step 2 satisfies the Palais-Smale condition. Therefore, by virtue of \cite[Theorem 2.7]{PR} $J$ has a critical point $u_0$. Moreover, $\displaystyle J(u_0)=\inf_W J$. Thus, $u\in W$  is a weak solution of \eqref{nonlinpde}. This completes the proof. \qed

  \bmhead{Acknowledgements}

 The second author  was  supported by the AMS-Simons Research Enhancement Grant for PUI Faculty.\\
 The first and third authors  were  supported by Facultad de Ciencias Exactas y Naturales, Universidad Nacional de Colombia, Sede Manizales – Facultad de Ciencias – Departamento de Matem\'aticas – Grupo de investigaci\'on: An\'alisis Matem\'atico AM de la Universidad Nacional de Colombia-Sede Manizales, through the following projects: Acerca de la Topolog\'i a Algebraica y sus aplicaciones a las Ecuaciones Diferenciales. Hermes code 63636. 
 Ciencia y Tecnolog\'ia para la Calidad en la Industria Licorera de Caldas (CalCIL): Un Enfoque desde la Facultad de Ciencias Exactas y Naturales, Hermes code  63945. Problemas en ecuaciones elípticas no lineales, Hermes code 63271. Support was also 
 received from 
 Ministerio de Ciencia Tecnolog\'ia en Innovaci\'on-Fondo Francisco Jos\'e de Caldas-Convocatoria de estancias post-doctorales de diplomacia cient\'ifica en el exterior para doctores colombianos 2022, code 928-Project: Estancia de Investigaci\'on Post Doctoral en Diplomacia Cient\'ifica con \'enfasis en Matem\'aticas.
\\

\textbf{Data availability:} The manuscript has no associated data.

\section*{Declarations}

\textbf{Conflict of interest:} The authors declare that they have no conflict of interest.


\bibliography{sn-bibliography}

\end{document}